\documentclass[12pt,onecolumn,letter]{article}
    \usepackage{amsmath}
    \usepackage{amsfonts}
    \usepackage{amssymb}
    \usepackage{amsthm}
    \usepackage{amstext}
    \usepackage{graphicx}
    \newcommand{\TT}{\ensuremath{ [T_{1},T_{2}] }}
    \newcommand{\ERRE}[1]{\ensuremath{\mathbb{R}^{#1}}}
    \newcommand{\ELEDOT}[2]{\ensuremath{\left\langle #1,#2 \right\rangle }}
    \newcommand{\ELEDOTT}[2]{\ensuremath{\big\langle #1,#2 \big\rangle_{\TT{}\times\Omega} }}
    \newcommand{\ELEDOTTT}[2]{\ensuremath{\big\langle #1,#2 \big\rangle_{Q_{T}} }}
    \newcommand{\HYPEQN}[3]{\ensuremath{
                            \big( -i\partial_{t} + {#1}_{0}(t,x)\big)^{2}{#3}
                            - \sum_{j=1}^{n} \big(-i\partial_{x_{j}} + {#1}_{j}(t,x) \big)^{2} {#3}
                            + {#2}(t,x){#3} }}
    \newcommand{\HYPOP}[2]{\ensuremath{
                            \big( -i\partial_{t} + {#1}_{0}(t,x)\big)^{2}
                            - \sum_{j=1}^{n} \big(-i\partial_{x_{j}} + {#1}_{j}(t,x) \big)^{2}
                            + {#2}(t,x) }}
    \newcommand{\HYPOPP}[2]{\ensuremath{
                            \big( -i\partial_{t} + {#1}_{0}\big)^{2}
                            - \sum_{j=1}^{n} \big(-i\partial_{x_{j}} + {#1}_{j} \big)^{2}
                            + {#2} }}
    \newcommand{\TRANSPOP}[1]{\ensuremath{
                            \big( -i\partial_{t} + {#1}_{0}(t,x)\big)
                            + \sum_{j=1}^{n} \omega_{j}\big(-i\partial_{x_{j}} + {#1}_{j}(t,x) \big)}}
    \newcommand{\TRANSPOPP}[1]{\ensuremath{
                            \big( -i\partial_{t} + {#1}_{0}\big)
                            + \sum_{j=1}^{n} \omega_{j}\big(-i\partial_{x_{j}} + {#1}_{j} \big)}}

    \newcommand{\DT}{\ensuremath{\partial_{t}}}
    
    \newcommand{\DXJ}[1]{\ensuremath{\partial_{x_{#1}}}}

    \newcommand{\POTENTIALS}[2]{\ensuremath{\big(\mathcal{#1} (t,x),#2 (t,x)\big)}}

    \theoremstyle{definition}
    \newtheorem{defn}{Definition}[section]

    \theoremstyle{plain}
    \newtheorem{lemma}{Lemma}[section]

    \theoremstyle{plain}
    \newtheorem{thm}{Theorem}[section]

    \theoremstyle{plain}
    \newtheorem{prop}[thm]{Proposition}

    \theoremstyle{plain}

    \theoremstyle{plain}

\begin{document}
    \title{Determination of time-dependent coefficients\\[1 ex]
          for a hyperbolic  inverse problem }
    \author{ Salazar, Ricardo. }

    \date{\today}

\maketitle

    \begin{abstract}
        We consider an inverse boundary value problem for the hyperbolic 
        partial differential equation
\begin{equation*}
(-i\partial_{t} + A_{0}(t,x))^{2}u(t,x)
- \sum_{j=1}^{n} \big(-i\partial_{x_{j}} + A_{j}(t,x)\big)^{2}u(t,x)
+ V(t,x)u(t,x) = 0
\end{equation*}
        with time dependent vector and scalar potentials $\big( \mathcal{A}=
        (A_{0},\dots,A_{m})$ and $V(t,x)$ respectively$\big)$ on a
        bounded, smooth cylindric domain $(-\infty,\infty)\times\Omega$. 
        Using a geometric optics construction we show that the boundary 
        data allows us to recover integrals of the potentials 
        along `light rays' and we then establish the uniqueness of these 
        potentials modulo a gauge transform. Also, 
        a logarithmic stability estimate is obtained 
        and the presence of obstacles inside the domain is studied.
        In this case, it is shown that under some geometric restrictions
        similar uniqueness results hold.
    \end{abstract}

\section{Introduction}
%\section{Introduction}
     Let $\Omega$ be a bounded domain in $\ERRE{n}$, $n\ge 2$,
     consider the hyperbolic equation with time dependent coefficients
     \begin{equation}\label{hyp:eqn}
     \HYPEQN{A}{V}{u} = 0 \quad\text{in }\ERRE{}\times\Omega,
     \end{equation}
     where $V(t,x)$, $A_{j}(t,x)$, $0\le j \le n$, are smooth functions vanishing
     when $\{|x| > R\}$ for some $R>0$. 
     The smooth vector field
     $\mathcal{A}(t,x)=
     (A_{0}(t,x),\dots,A_{n}(t,x))$ is called the \emph{vector potential}, the
     function $V(t,x)$ is called the \emph{scalar potential} and 
     equation (\ref{hyp:eqn}) is often referred to as the \emph{relativistic Schr\"odinger
     equation} (see \cite{Schiff}).
     
     For the above differential equation
     we impose the initial and boundary conditions
    \begin{align}\label{hyp:eqn1a}
     u(t,x) = \partial_{t}u(t,x) & = 0\makebox[2.5 em]{}
                                      \quad\text{for }t << 0 \\ \label{hyp:eqn1b}
                          u(t,x) & = f\makebox[2.5 em]{$(t,x)$}
                                      \quad\text{on }\ERRE{}\times\partial\Omega,
     \end{align}
     where $f$ is a compactly supported smooth function on $\ERRE{}\times\partial\Omega$.
     Solutions to (\ref{hyp:eqn}) satisfying (\ref{hyp:eqn1a}) and (\ref{hyp:eqn1b})
     are unique and we can define the \emph{Dirichlet to Neumann} operator by
     \begin{equation}\label{D:to:N}
     \Lambda(f) := \left( \partial_{\nu} + i A(t,x)\cdot\nu
     \right)u(t,x) \Big|_{\ERRE{}\times\partial\Omega}
     \end{equation}
     where $u$ is the solution of (\ref{hyp:eqn})-(\ref{hyp:eqn1b}), $\nu$ is the
     exterior unit normal to $\partial\Omega$ and we have set $A(t,x) =
     (A_{1}(t,x),\dots,A_{n}(t,x))$.
     The \emph{Inverse Boundary Value Problem} is the recovery of $\mathcal{A}(t,x)$
     and $V(t,x)$ knowing $\Lambda(f)$ for all $f\in C_{0}^{\infty}\big(\ERRE{}
     \times \partial\Omega \big)$.

      Inverse problems is a topic in mathematics that has been growing in interest for the
      past decades, in part, due to its wide range of applications, from medicine
      to acoustics to electromagnetism just to mention a few (see for instance \cite{Isakov}
      for some of the latest tools and techniques employed in the solutions of these problems).
      In the case of the hyperbolic inverse boundary value problem (1)-(4) with time 
      independent coefficients, a powerful tool called the \emph{boundary control method}, 
      or BC-method for short,
      was discovered by Belishev (see \cite{Belishev}). It was later developed by Belishev,
      Kurylev, Lassas, and others (\cite{Kurylev},\cite{Kurylev:Lassas}), and more recently
      a new approach to this problem based on the BC-method was developed by  Eskin 
      in (\cite{Eskin:approach},\cite{Eskin:approach2}). On a similar note, Stefanov and Uhlmann 
      established uniqueness and stability results for the wave equation in anisotropic 
      media (see \cite{Stefanov:Uhlmann} and \cite{Uhlmann} for a survey of these results). 

      Nevertheless, the case of time dependent coefficients has seen very 
      little progress in recent years. In the case of the vector potential 
      being identically equal to zero ($\mathcal{A}\equiv 0$ in (1)), Stefanov 
      \cite{Stefanov} and Ramm-Sj\"ostrand \cite{Ramm:Sjostrand}, have 
      shown that the Dirichlet to Neumann map completely determines the
      scalar potentials. More recently, Eskin \cite{Eskin:time:dependent}
      considered the case of time-dependent potentials that are analytic in time. 
      The analiticity of the time variable is related to the use of a
      unique continuation theorem
      established by Tataru in \cite{Tataru}. In this paper we eliminate 
      the restriction on the analiticity in the time variable and not 
      only we extend the uniqueness results in \cite{Stefanov}, 
      \cite{Ramm:Sjostrand}, but we also establish a logarithmic stability 
      estimate for the case when the vector potentials are compactly 
      supported in space and time.
      
      We also study the problem 
      with obstacles inside the domain and show that under some geometric
      considerations similar uniqueness results hold. The presence of 
      these obstacles in the domain may lead to the Aharonov-Bohm effect.
      This problem was considered by Nicoleau and Weder 
      in the context of the inverse scattering 
      (see \cite{Nicoleau} and \cite{Weder} respectively),
      and by Eskin \cite{Eskin:obstacles} in the context
      of the inverse boundary value problem for the Scrh\"odinger equation.

    This work is structured as follows. In section \ref{sec:non:uniq} we introduce the
    notion of \emph{gauge equivalent} for a pair of vector and scalar potentials and we
    make some remarks about uniqueness.
    In section \ref{sec:go:solns} we construct geometric
    optic solutions (GO for short) for equation (\ref{hyp:eqn})
    satisfying the set of initial conditions (\ref{hyp:eqn1a}). In section
    \ref{sec:green:formula} we establish
    a Green's formula  for these types of problems and show that the light ray
    transforms of gauge equivalent potentials agree.
    In section \ref{sec:proofs:part1} we prove uniqueness of the potentials
    in the case where no obstacles are allowed inside the domain $\Omega$. In 
    section \ref{sec:stability:estimates}, based on the works of 
    Isakov \cite{Isakov:uniq:stab}, Isakov and Sun \cite{Isakov:Sun} and more 
    recently Begmatov \cite{Begmatov}, we establish a stability
    result of logrithmic type for the particular case when the potentials
    are compactly supported in both space and time. Finally in
    section \ref{sec:proofs:part2} we consider the
    problem when one or more convex bodies are allowed inside the domain (by
    imposing a geometric restriction on the layout of these obstacles).

%\section{Gauge Equivalence}\label{sec:non:uniq}
    \section{Gauge equivalence}\label{sec:non:uniq}
    The ultimate goal in most inverse boundary value problems is the recovery
    of the coefficients of a partial differential equation, however for application
    purposes this recovery is meaningless unless it can be done in some sort of `unique'
    way. In our case this type of uniqueness is obtained modulo a \emph{gauge transform}.
    %(see definition below) \\
% 
  \begin{defn}\label{def:gauge:equiv}
    We say that the vector and scalar potentials $\POTENTIALS{A}{V}$ and
    $\POTENTIALS{A'}{V'}$ are \emph{gauge equivalent} if there exists $g(t,x) \in
    C^{\infty}(\ERRE{}\times\overline{\Omega}))$ such that $g(t,x)\not = 0$
    on $\ERRE{}\times\overline{\Omega})$, $g=1$ on $\ERRE{}\times\partial\Omega$ and
    \begin{align*}
     \mathcal{A}'(t,x)= & \mathcal{A}(t,x) - \frac{i}{g(t,x)} \nabla_{t,x} g(t,x) \\
     V'(t,x) = & V(t,x),
    \end{align*}
    where $\nabla_{t,x} :=(\partial_{t},\partial_{x})=
    (\partial_{t},\partial_{x_{1}},\dots,\partial_{x_{n}})$ is the $(n+1)$-dimensional
    gradient.
    The mapping $(\mathcal{A},V)\to(\mathcal{A}^{\prime},V^{\prime})$
    is called a \emph{gauge transform}.
  \end{defn}
    The definition above includes the more general case when obstacles are present
    inside the domain. When $\Omega$ is simply connected (no obstacles), the
    gauge $g$ has the particular form $g(t,x)=e^{i\varphi(t,x)}$ where
    $\varphi(t,x)\in C^{\infty}(\ERRE{}\times\Omega)$.
    Then $-\frac{i}{g(t,x)}\nabla_{(t,x)}g(t,x)=\nabla_{(t,x)}\varphi(t,x)$
    and we see that two vector potentials are gauge equivalent if their
    difference is the gradient of a smooth function.
    The following proposition tells us that recovery of the potentials
    can only be done up to a gauge transform.
    \begin{prop}\label{prop1}
    If $u(t,x)$ is a solution of (\ref{hyp:eqn})-(\ref{hyp:eqn1b}) and $g(t,x)$ is as in
    definition (\ref{def:gauge:equiv}), then $v(t,x)=g(t,x)u(t,x)$ satisfies
%{\small
     \begin{align}\label{eq:1.7}
          \HYPEQN{A'}{V'}{v} = 0 & \thickspace\text{in }\ERRE{}\times\Omega \\ \nonumber
          v=\partial_{t} v = 0 & \thickspace\text{for }t << 0 \\ \nonumber
          v = f g \big|_{\ERRE{}\times\partial\Omega} &
                      \thickspace\text{on } \ERRE{}\times\partial\Omega .
     \end{align}
%}
    with $(\mathcal{A}',V')$ and $(\mathcal{A},V)$ gauge equivalent.\\
    In addition if $\Lambda'$ is the Dirichlet to Neumann operator associated to
    (\ref{eq:1.7}),
    then
    \begin{equation}
       \Lambda'\big(v |_{\ERRE{}\times\partial\Omega}\big) =
       g |_{\ERRE{}\times\partial\Omega}
       \Lambda\big(u |_{\ERRE{}\times\partial\Omega}\big)
    \end{equation}
    i.e., $\Lambda' = \Lambda$ since $g|_{\ERRE{}\times\partial\Omega}=1$.
    \end{prop}
    \begin{proof}
    Setting $x_{0}= t$ we see that for $0\le j\le n$,
    \begin{equation*}
       \left( -i\partial_{x_{j}} + A'_{j}(t,x)\right) v(t,x) =
       g(t,x) \left( -i\partial_{x_{j}} + A'_{j}(t,x)
       -\frac{i}{g(t,x)}\partial_{x_{j}}g(t,x) \right) u(t,x).
    \end{equation*}
    Choosing $A'_{j}(t,x)$ so that $A'_{j}=A_{j} + \frac{i}{g}\partial_{x_{j}}g$
    for $0\le j \le n$, we get
    \begin{align*}
       \left( -i\partial_{x_{j}} + A'_{j}(t,x)\right)^{2}v(t,x)
       & = \left( -i\partial_{x_{j}} + A'_{j}(t,x)\right)\big( g(t,x)
           \left( -i\partial_{x_{j}} + A_{j}(t,x)\right)u(t,x) \big) \\
       & = g(t,x) \left( -i\partial_{x_{j}} + A_{j}(t,x)\right)^{2} u(t,x),
    \end{align*}
    thus
    \begin{align*}
    \HYPEQN{A'}{V'}{v} & = \\
    g(t,x)\Big(\HYPEQN{A}{V}{u}\Big) &  = 0
    \end{align*}
    as $u$ is a solution of (\ref{hyp:eqn}).
    Also notice that since $g$ is smooth and $u$ satisfies (\ref{hyp:eqn1a}) and
    (\ref{hyp:eqn1b}) we have for $t<<0$
    \begin{align*}
       v(t,x) & = u(t,x)g(t,x)= 0 \\
       \partial_{t}v(t,x) & = u(t,x)\partial_{t}g(t,x) + \partial_{t}u(t,x)
       g(t,x) = 0,
    \end{align*}
    similarly $v(t,x)\big|_{\ERRE{}\times\partial\Omega} =
    \big(g(t,x)u(t,x)\big)\big|_{\ERRE{}\times\partial\Omega}
    = fg\big|_{\ERRE{}\times\partial\Omega}$.
    To conclude we simply notice that
    \begin{align*}
       \Lambda'\left(v\big|_{\ERRE{}\times\partial\Omega}\right) & =
       \big(\partial_{\nu} (g u) + i A'\cdot \nu (g u )\big)
       \Big|_{\ERRE{}\times\partial\Omega}\\
       & = \Big(\left(\partial_{\nu} g \right)u + g\left(\partial_{\nu} u \right)
         + i \left(A + ig^{-1}\partial_{x}g \right)\cdot \nu (gu)\Big)
         \Big|_{\ERRE{}\times\partial\Omega}\\
       & = g \big(\partial_{\nu} u + i(A\cdot \nu)u\big)
             \Big|_{\ERRE{}\times\partial\Omega}
         + \big( (\partial_{\nu}g)u - (\partial_{\nu}g)u\big)
         \Big|_{\ERRE{}\times\partial\Omega} \\
       & = g\big|_{\ERRE{}\times\partial\Omega}
       \Lambda\left( u \big|_{\ERRE{}\times\partial\Omega}\right)
    \end{align*}
    \end{proof}
    If the above equality holds we shall say that the Dirichlet to
    Neumann maps $\Lambda$ and $\Lambda'$ are gauge equivalent.
    Summarizing, we have shown that if the vector and scalar potentials are gauge
    equivalent then the Dirichlet to Neumann maps are equal. In the following pages we
    shall attempt to prove the converse, roughly speaking: 
    \emph{If for a pair of vector and scalar potentials the Dirichlet to Neumann
    operators associated to the hyperbolic equation (\ref{hyp:eqn})-(\ref{hyp:eqn1b})
    are equal, then so are the vector and scalar potentials.}

%\section{Geometric optics}\label{sec:go:solns}
\section{Geometric optics}\label{sec:go:solns}
    For the hyperbolic problem (\ref{hyp:eqn})-(\ref{hyp:eqn1b}) we shall attempt
    to construct geometric optics solutions supported near light rays. In order for
    us to achieve this goal we consider solutions 
    having the form
    \begin{equation}\label{geom:optic:sols}
    u(t,x) = e^{ik(t-\omega\cdot x)}
             \sum_{p=0}^{N}\frac{v_{p}(t,x)}{(2ik)^{p}}
            + v^{(N+1)}(t,x),\qquad \omega\in S^{n-1}, k\in\mathbb{R}.
    \end{equation}
    For $u$ as above we have
    \begin{equation}
      \left(-i\DT + A_{0}\right) u = e^{ik(t-\omega\cdot x)}
      \left( -i\DT + A_{0} + ik \right) \big( \sum_{p=0}^{N}\frac{v_{p}}{(2ik)^{p}}
            +  e^{-ik(t-\omega\cdot x)}v^{(N+1)} \big)
    \end{equation}
    applying the above identity twice to a solution of (\ref{hyp:eqn})
    we get
    \begin{equation*}
    \begin{split}
    \big(-i\DT + A_{0}\big)^{2} u
    & = \big(-i\DT + A_{0}\big)
      \Big( e^{ik(t-\omega\cdot x)} \big( -i\DT + A_{0} + ik \big)
      \big(\sum_{p=0}^{N} \frac{v_{p}}{(2ik)^{p}} \thickspace + \\
    & \qquad  e^{-ik(t-\omega\cdot x)}v^{(N+1)}\big) \Big) \\
    & = e^{ik(t-\omega\cdot x)} \big( -i\DT + A_{0} + ik \big)
        \big(-i\DT + A_{0} + ik\big)
        \Big(\sum_{p=0}^{N} \frac{v_{p}}{(2ik)^{p}} \thickspace + \\
    & \qquad e^{-ik(t-\omega\cdot x)}v^{(N+1)}\Big) 
    \end{split}
    \end{equation*}
    this is
    \begin{multline*}
      \big(-i\DT + A_{0}\big)^{2} u 
       = e^{ik(t-\omega\cdot x)} \Big( \big(-i\DT +A_{0}\big)^{2} + \\
          2ik\big(-i\DT + A_{0}\big) -k^{2} \Big)
          \Big(\sum_{p=0}^{N} \frac{v_{p}}{(2ik)^{p}} \thickspace + 
                 e^{-ik(t-\omega\cdot x)}v^{(N+1)}\Big).
    \end{multline*}
    Since a similar formula holds for  $1\le j\le n$, we obtain
    \begin{multline*}
         \big(-i\DXJ{j} + A_{j}\big)^{2} u 
	  = e^{ik(t-\omega\cdot x)} \Big( \big(-i\DXJ{j} +A_{j}\big)^{2} + \\
             2ik\omega_{j}\big(-i\DXJ{j} + A_{j}\big) - k^{2}\omega_{j}^{2} \Big)
             \Big(\sum_{p=0}^{N} \frac{v_{p}}{(2ik)^{p}} \thickspace + 
	         e^{-ik(t-\omega\cdot x)}v^{(N+1)}\Big),
    \end{multline*}
    thus equation (\ref{hyp:eqn}) becomes
    \begin{align}\nonumber
    0= L u % = & \thinspace L\big(e^{ik(t-\omega\cdot x)}v \big) \\ \nonumber
           = & \thinspace e^{ik(t-\omega\cdot x)}
               \Big(\HYPOPP{A}{V}\Big)v \\ \nonumber
             & \thinspace + 2ik e^{ik(t-\omega\cdot x)}
               \Big(\TRANSPOPP{A}\Big)v \\ \nonumber
             & \thinspace + e^{ik(t-\omega\cdot x)}
               \Big(-k^{2} + \sum_{j=1}^{n}(\omega_{j}k)^{2}\Big)
                  v \\ \label{hyp:eqn:for:u1}
    0=L u = & \thinspace e^{ik(t-\omega\cdot x)} \big(L +2ik\mathcal{L} \big) v,
    \end{align}
    where we have set
    \begin{align}\label{geom:series:soln}
    v(t,x) & = \sum_{p=0}^{N} \frac{v_{p}(t,x)}{(2ik)^{p}} + e^{-ik(t-\omega\cdot x)}v^{(N+1)}(t,x) \\
    L & = \HYPOP{A}{V} \\
    \mathcal{L} & = \TRANSPOP{A}.
    \end{align}
    Plugging in the expression for $v$ into (\ref{hyp:eqn:for:u1}) we obtain
    \begin{equation}
    0 = \big( 2ik\mathcal{L} + L \big)
        \left( v_{0} + \frac{1}{(2ik)}v_{1}+\dots+\frac{1}{(2ik)^{N}}v_{N}
                   + e^{-ik(t-\omega\cdot x)} v^{(N+1)}\right),
    \end{equation}
    which in turn can be rewritten as
    \begin{align}\nonumber
    (2ik)\mathcal{L}v_{0} & + \big( \mathcal{L}v_{1} + L v_{0} \big) +
    \frac{1}{(2ik)}\big( \mathcal{L}v_{2} + L v_{1} \big) + \dots + \\
    & \frac{1}{(2ik)^{N-1}}\big( \mathcal{L}v_{N} + L v_{N-1} \big) +
    \frac{1}{(2ik)^{N}}L v_{N} + e^{-ik(t-\omega\cdot x)}L v^{(N+1)} = 0,
    \end{align}
    where we have used the identity $(2ik\mathcal{L} + L)\big( e^{-ik(t-\omega\cdot x)}v^{(N+1)}\big)
    =e^{-ik(t-\omega\cdot x)} L v^{(N+1)}$.

    Notice that a solution of (\ref{hyp:eqn})-(\ref{hyp:eqn1a})
    can be found by solving the $N+1$ transport equations
    \begin{equation}
      \mathcal{L} v_{0} = 0,\qquad \mathcal{L} v_{j}
                        = -L v_{j-1}, \qquad 1\le j \le N
    \end{equation}
    with initial conditions supported near a neighborhood of the light ray
    $\gamma = \{ (t^{\prime},x^{\prime}) + s(1,\omega)\thinspace : \thinspace
    (t^{\prime},x^{\prime}) \perp (1,\omega), s \in \ERRE{}\}$
    (we assume that $\gamma$ intersects the plane $t=T_{1}$ outside
    of the cylinder $\ERRE{}\times\Omega$;
    as well as the second order equation
    \begin{equation}\label{hyp:eqn:after:go:solns}
    L v^{(N+1)} = -\frac{e^{ik(t-\omega\cdot x)}}{(2ik)^{N+1}} L v_{N}
    \end{equation}
    with initial and boundary conditions
    \begin{align*}
     v^{(N+1)}(t,x) & = 0
                    & \text{for }t = T_{1} \quad \phantom{nx\in\partial\Omega}\\
     \partial_{t}v^{(N+1)}(t,x) & = 0
                    & \text{for }t = T_{1} \quad \phantom{nx\in \partial\Omega}\\
     v^{(N+1)}(t,x) & = 0
                    & \text{for }t \ge T_{1}, \quad x\in\partial\Omega.
     %\nu\cdot\nabla_{x}v^{(N+1)}(t,x) & = 0
     %               & \text{on }\ERRE{}\times\partial\Omega,
     \end{align*}

    %\begin{align*}
    % v^{(N+1)}(t,x) & = -\frac{1}{2ik}\sum_{p=0}^{N} \frac{v_{p}(t,x)}{(2ik)^{p}}
    %                & \text{for }t\le T_{1} \\
    % \partial_{t}v^{(N+1)}(t,x) & = -\frac{1}{2ik}\sum_{p=0}^{N} \frac{\partial_{t}v_{p}(t,x)}{(2ik)^{p}}
    %                & \text{for }t\le T_{1} \\
    % \nu\cdot\nabla_{x}v^{(N+1)}(t,x) & = 0
    %                & \text{on }\ERRE{}\times\partial\Omega,
    % \end{align*}
    %
    The above differential equation has a unique solution. Moreover if $h$ is
    the right hand side of (\ref{hyp:eqn:after:go:solns}) we have
    (see for instance Isakov \cite{Isakov}, pp. 185) that if $T_{1}<t<T$
    and $k>1$
    \begin{align}\nonumber
     ||v^{(N+1)}(t,\thinspace)||_{L^{2}(\Omega)}
       & \le C
       ||h||_{L^{2}((T_{1},T)\times\Omega)} \\ \label{power:of:k}
       & \le \frac{C}{k^{N}}
    \end{align}
    %\begin{align}\nonumber
    % ||v^{(N+1)}(t,\thinspace)||_{L^{2}(\Omega)}
    %   & \le C \left(
    %   ||h||_{L^{2}((T_{1},T)\times\Omega)} +
    %   ||v(T_{1},\thinspace)^{(N+1)}||_{H^{1}(\Omega)} +
    %   ||\partial_{t}v(T_{1},\thinspace)^{(N+1)}||_{L^{2}(\Omega)}
    %   \right) \\ \label{power:of:k}
    %   & \le C \left( \frac{1}{k^{N+1}} + \frac{2}{k}\right) \le \frac{C}{k}
    %\end{align}
    Thus, we have shown that we can find a solution
    $u=e^{ik(t+\omega\cdot x)}(v_{0}+\mathcal{O}(k^{-1}))$
    of (\ref{hyp:eqn}) satisfying the set of initial conditions
    (\ref{hyp:eqn1a}). Let us now examine the first term
    in (\ref{geom:series:soln}) by solving the transport equation
    \begin{equation}\label{eqn:2.23}
       0 = \mathcal{L}v_{0}(t,x) = \sum_{j=0}^{n}\omega_{j}\partial_{x_{j}}v_{0}(t,x) +
       i \sum_{j=0}^{n}\omega_{j}A_{j}(t,x)v_{0}(t,x)
    \end{equation}
    where we have set $\omega_{0}=1$ and $\partial_{x_{0}}=\partial_{t}$.

    Equation (\ref{eqn:2.23}) is a first order transport equation that
    can be solved by the method of the characteristics or by performing
    a change of variables that turns the PDE into an ordinary differential
    equation. Either way, the solution we obtain is given by
    \begin{equation}\label{eqn:2.24}
       v_{0}(t,x)=\chi(t,x)
       \exp \left( -i
          \int_{-\infty}^{\frac{1}{2}(t+\omega\cdot x)} 
	       \sum_{j=0}^{n} \omega_{j}A_{j}(t^{\prime}+s,x^{\prime}+s\omega)
	       \thinspace\mathrm{d}s,
       \right)
    \end{equation}
    where $(t^{\prime},x^{\prime})=(t,x)-\frac{1}{2}(t+\omega\cdot x)(1,\omega)$ is
    the projection of $(t,x)$ into $\Pi_{(1,\omega)}$, the hyperplane perpendicular
    to $(1,\omega)$ and $\chi$ is any function that is constant
    along the direction given by $(1,\omega)$
    and whose support is contained in
    a neighborhood of the light ray $\gamma=\{
    (t',x') + s(1,w)\thinspace | \thinspace s\in\ERRE{}\}$ 
    (in general $\chi$ can be complex valued but for our 
    purposes we will assume it is real valued).

    Summarizing, we have been able to construct a solution of (\ref{hyp:eqn})-(\ref{hyp:eqn1a})
    having the form:
    \begin{equation}\label{GO1a}
    u(t,x) = e^{ik(t-\omega\cdot x)- i R_{1}(t,x;\omega)}
             \big( \chi(t',x') + \mathcal{O}(k^{-1}) \big),
    \end{equation}
    where
    \begin{equation}\label{GO1b}
    R_{1}(t,x;\omega) = \int^{\frac{1}{2}(t+\omega\cdot x)}_{-\infty}
                 \sum_{j=0}^{n}\omega_{j}A_{j}(t^{\prime}+s,x^{\prime}+s\omega)
             \thinspace\mathrm{d}s
    \end{equation}

    In a similar way one can obtain geometric optics solutions for the backwards hyperbolic
    problem
    \begin{align*}
      L v & = 0\qquad\text{in }(-\infty,T_{2})\times\Omega \\
      v = \DT v & =0\qquad\text{for }t = T_{2}
      %v & =g\qquad\text{on } \ERRE{}\times\partial\Omega .
    \end{align*}

    In the following section we will derive a Green's Formula for these kinds
    of hyperbolic operators and will use the Geometric Optics representations
    to conclude that the Dirichlet to Neumann data determines the vectorial
    and scalar ray transforms of the potentials along \emph{`light rays'} (this
    is, rays that make a 45 degree angle with the hyperplane $t=0$).

%\section{Green's formula}\label{sec:green:formula}
\section{Green's formula}\label{sec:green:formula}

   This technique has had a lot of success in the context of inverse problems, 
   in particular for the case of elliptic problems, the fundamental paper of 
   Sylvester and Uhlmann \cite{Sylvester:Uhlmann} has been a source of 
   inspiration for several other uniqueness results (see also Isakov's review 
   paper \cite{Isakov:uniq:stab} for more information on this subject).
   
   For $T_{1}$ and $T_{2}$ two real numbers with $T_{1}<T_{2}$ we
   consider the forward and backward hyperbolic equations
   \begin{center}
   \begin{tabular}{r l c r l}
    $L_{1} u = 0$ &$\quad\text{in }\TT{}\times\Omega$ & $\qquad$&
    $L_{2}^{*} v = 0$ &$\quad\text{in }\TT{}\times\Omega$ \\
    $u = \DT u = 0$ &$\quad\text{for } t = T_{1} $ & $\quad$&
    $v = \DT v = 0$ &$\quad\text{for } t = T_{2} $ \\
    $u=f$           &$\quad\text{on }\TT{}\times\partial\Omega$ & $\quad$ &
    $v=g$           &$\quad\text{on }\TT{}\times\partial\Omega$,
   \end{tabular}
   \end{center}
   where we have set
   \begin{align*}
      L_{1} = L(\mathcal{A}^{(1)},V^{(1)})
      & = \big( -i\DT + A_{0}^{(1)}(t,x)\big)^{2} -
          \sum_{j=1}^{n} \big(-i\DXJ{j} + A_{j}^{(1)}(t,x) \big)^{2} +
          V^{(1)}(t,x) \\
      L_{2}^{*} = L(\overline{\mathcal{A}^{(2)}},\overline{V^{(2)}})
      & = \big( -i\DT + \overline{A_{0}^{(2)}(t,x)}\big)^{2} -
          \sum_{j=1}^{n} \big(-i\DXJ{j} + \overline{A_{j}^{(2)}(t,x)} \big)^{2} +
          \overline{V^{(2)}(t,x)}
   \end{align*}
   and let us assume that the Dirichlet to Neumann operators
   \begin{align}\label{DN1}
      \Lambda_{1}(f)
      & = \big(\partial_{\nu} + i\nu\cdot A^{(1)}(t,x)\big)u(t,x)\big|_{\TT{}
          \times\partial\Omega}\\ \label{DN2}
      \Lambda_{2}(g)
      & = \big(\partial_{\nu} + i\nu\cdot A^{(2)}(t,x)\big)v(t,x)\big|_{\TT{}
          \times\partial\Omega}
   \end{align}
   equal on $(T_{1},T_{2})\times\partial\Omega$, i.e., $\Lambda_{1}f = \Lambda_{2}f$
   for all $f$ smooth and supported on the set $(T_{1},T_{2})\times\partial\Omega$.

   \textbf{Remark:} \textit{Notice that for the operator $L^{*}_{2}$ we associate the
   Dirichlet to Neumann map
   \begin{equation*}
      \Lambda_{2}^{*}(g)
       = \big(\partial_{\nu} + i\nu\cdot \overline{A^{(2)}(t,x)}\big)v(t,x)\big|_{\ERRE{}
          \times\partial\Omega}
   \end{equation*}
   and that our main assumption is $\Lambda_{1} = \Lambda_{2}$ on
   $\ERRE{}\times\partial\Omega$. This is no mistake as later on we will
   show that our notation is justified as the $\mathrm{L}^{2}$ adjoint
   of $\Lambda_{2}$ is indeed $\Lambda^{*}_{2}$.}

   Denoting by $\langle\text{ , }\rangle_{\TT{}\times\Omega}$,
   $\langle\text{ , }\rangle_{\Omega}$ and
   $\langle\text{ , }\rangle_{\TT{}\times\partial\Omega}$ the $L^{2}$ inner products in
   $\TT{}\times\Omega$, $\Omega$ and $\TT{}\times\partial\Omega$ respectively we obtain the
   following integration by parts formulas for $A^{(1)}_{0}$
   \begin{align*}
     \big\langle\big( -i\DT + A_{0}^{(1)}\big)^{2}u,v
                 \big\rangle _{[T_{1},T_{2}]\times\Omega} =
     & \thinspace\big\langle\big( -i\DT+A_{0}^{(1)}\big)u, \big(-i\DT+
                  \overline{A_{0}^{(1)}}\big)v 
                 \big\rangle_{[T_{1},T_{2}]\times\Omega} \\
     & \thinspace - i\big\langle\big(-i\DT+A_{0}^{(1)}\big)
                  u(t,\cdot), v(t,\cdot)\big\rangle_{\Omega}\Big|_{T_{1}}^{T_{2}}
   \end{align*}
   where $\nu=(\nu^{(1)},\dots,\nu^{(n)})$ is the exterior unit normal to $\partial\Omega$
   and $u$, $v$ are solutions of the forward and backward hyperbolic equations respectively.

   In view of the initial conditions we obtain
   \begin{equation}\label{eqn:3.3}
     \big\langle\big( -i\DT + A_{0}^{(1)}\big)^{2}u, v
     \big\rangle _{[T_{1},T_{2}]\times\Omega} =
     \big\langle\big(-i\DT+A_{0}^{(1)}\big)u, \big(-i\DT
                   +\overline{A_{0}^{(1)}}\big)v
     \big\rangle _{[T_{1},T_{2}]\times\Omega}.
   \end{equation}
   Also for $j=1,\dots,n$ we have for $A^{(1)}_{j}$
   \begin{align*}
     \big\langle\big( -i\DXJ{j} + A_{j}^{(1)}\big)^{2}u, v
     \big\rangle _{[T_{1},T_{2}]\times\Omega} =
     & \thinspace
       \big\langle\big( -i\DXJ{j}+A_{j}^{(1)}\big)u, \big( -i\DXJ{j}
                          +\overline{A_{j}^{(1)}} \big)v 
       \big\rangle _{[T_{1},T_{2}]\times\Omega} \\
     & \thinspace - i
       \big\langle\big( -i\DXJ{j}+A_{j}^{(1)} \big)
                          u\nu^{(j)}, v
       \big\rangle_{\TT{}\times\partial\Omega},
   \end{align*}
   hence
   \begin{multline}
     \sum_{j=1}^{n} \big\langle\big( -i\DXJ{j} + A_{j}^{(1)} \big)^{2}u,v
                    \big\rangle_{\TT{}\times\Omega}  = \\
      \sum_{j=1}^{n}\big\langle\big( -i\DXJ{j}+A_{j}^{(1)}\big)u, \big( -i\DXJ{j}
                    +\overline{A_{j}^{(1)}} \big) v 
                    \big\rangle_{\TT{}\times\Omega} \\
      - \ELEDOT{\Lambda_{1}\left(f\right)}{g}_{\TT{}\times\partial\Omega},
   \end{multline}
   where $f=u|_{\TT\times\partial\Omega}$ and
   $g=v|_{\TT\times\partial\Omega}$.

   Similarly for 
   %the differential operator 
   $L_{(2)}^{*}$ we have
   \begin{equation}
     \big\langle u, \big( -i\DT + \overline{A_{0}^{(2)}} \big)^{2}v 
     \big\rangle_{\TT{}\times\Omega}  = 
     \big\langle\big( -i\DT+A_{0}^{(2)} \big)u, \big(-i\DT+\overline{A_{0}^{(2)}} \big)v
     \big\rangle_{\TT{}\times\Omega} 
   \end{equation}
   and
   \begin{multline}\label{eqn:3.6}
     \sum_{j=1}^{n}
          \big\langle u, \big( -i\DXJ{j} + \overline{A_{j}^{(2)}} \big)^{2}v 
          \big\rangle_{\TT{}\times\Omega}  = \\
     \sum_{j=1}^{n}
          \big\langle\big( -i\DXJ{j}+A_{j}^{(2)} \big)u, \big( -i\DXJ{j}
                    +\overline{A_{j}^{(2)}} \big)v 
          \big\rangle_{\TT{}\times\Omega} \\
      - \ELEDOT{f}{\Lambda_{2}^{*}\left(g\right)}_{\TT{}\times\partial\Omega}.
   \end{multline}

   Combining expressions (\ref{eqn:3.3})-(\ref{eqn:3.6}) and recalling that $u$ and
   $v$ are solutions of the forward and backward hyperbolic problem respectively we obtain
   \begin{align}\nonumber
      0= &\thinspace
      \ELEDOTT{L_{1}u}{v} - \ELEDOTT{u}{L_{2}^{*}v} \\ \nonumber
      = & \thinspace
      \underbrace{
        \ELEDOTT{\big(-i\DT+A_{0}^{(1)}\big)u}{\big(-i\DT+\overline{A_{0}^{(1)}}\big)v}
      }_{I_{1}} \\ \nonumber
      & \thinspace -
      \underbrace{
        \ELEDOTT{\big(-i\DT+A_{0}^{(2)}\big)u}{\big(-i\DT+\overline{A_{0}^{(2)}}\big)v}
      }_{I_{2}} \\ \nonumber
      & \thinspace -
      \underbrace{ \sum_{j=1}^{n}
        \ELEDOTT{\big(-i\DXJ{j}+A_{j}^{(1)}\big)u}{\big(-i\DXJ{j}
                 +\overline{A_{j}^{(1)}}\big)v}
      }_{I_{3}} \\ \nonumber
      & \thinspace
      +
      \underbrace{ \sum_{j=1}^{n}
        \ELEDOTT{\big(-i\DXJ{j}+A_{j}^{(2)}\big)u}{\big(-i\DXJ{j}
                 +\overline{A_{j}^{(2)}}\big)v}
      }_{I_{4}} \\ \nonumber
      & \thinspace
      + \ELEDOTT{V^{(1)}u}{v} - \ELEDOTT{V^{(2)}u}{v} \\ \label{eqn:3.7}
      & \thinspace
      + \ELEDOT{f}{\Lambda_{2}^{*}\big(g\big)}_{\TT{}\times\partial\Omega}
         - \ELEDOT{\Lambda_{1}\big(f\big)}{g}_{\TT{}\times\partial\Omega}
   \end{align}
   Let us now study the terms $I_{1}$, $I_{2}$, $I_{3}$ and $I_{4}$ 
   appearing in the above formula.
   For $I_{1}$ and $I_{2}$ we have
   \begin{align*}
   I_{1} + I_{2} =
     & \thinspace
        \ELEDOTT{\big(-i\DT+A_{0}^{(1)}\big)u}{\big(-i\DT+\overline{A_{0}^{(1)}}\big)v} \\
     %& \thinspace
      - & \ELEDOTT{\big(-i\DT+A_{0}^{(2)}\big)u}{\big(-i\DT+\overline{A_{0}^{(2)}}\big)v} \\
   = & \thinspace
          \Big\langle -i\DT u,-i\DT v\Big\rangle_{\ERRE{}\times\Omega}
        + \ELEDOTT{-i\DT u}{\overline{A_{0}^{(1)}}v}
        + \ELEDOTT{A_{0}^{(1)}u}{-i\DT v} \\
     %& \thinspace
      + & \ELEDOTT{A_{0}^{(1)}u}{\overline{A_{0}^{(1)}}v} 
        - \Big\langle -i\DT u,-i\DT v \Big\rangle_{\ERRE{}\times\Omega}
        - \ELEDOTT{-i\DT u}{\overline{A_{0}^{(2)}}v} \\
     %& \thinspace
      - & \ELEDOTT{A_{0}^{(2)}u}{-i\DT v}
        - \ELEDOTT{A_{0}^{(2)}u}{\overline{A_{0}^{(2)}}v},
   \end{align*}
   this is
   \begin{align}\nonumber
   I_{1} + I_{2} =
   & \thinspace
      \ELEDOTT{\big(A_{0}^{(1)}-A_{0}^{(2)}\big)u}{-i\DT v}
        + \ELEDOTT{\big(A_{0}^{(1)}-A_{0}^{(2)}\big)
                      \big(-i\DT u\big)}{v} \\ \label{eqn:3.16}
   & \thinspace
        + \ELEDOTT{\big(\big(A_{0}^{(1)}\big)^{2}
                 -\big(A_{0}^{(2)}\big)^{2}\big)u}{v}
  \end{align}
   and a similar computation shows that
   \begin{multline} \label{eqn:3.17}
   I_{2} + I_{3} =
    \thinspace \sum_{j=1}^{n} 
     %\big(
      \ELEDOTT{\big(A_{j}^{(1)}-A_{j}^{(2)}\big)u}{-i\DXJ{j} v} \\
    + \thinspace \sum_{j=1}^{n} 
         \ELEDOTT{\big(A_{j}^{(1)}-A_{j}^{(2)}\big)
               \big(-i\DXJ{j} u\big)}{v} \\
     %\big) \\ 
    \thinspace + \sum_{j=1}^{n}
          \ELEDOTT{\big( (A_{j}^{(1)} )^{2}
                 - (A_{j}^{(2)} )^{2}\big)u}{v} .
  \end{multline}

   Combining (\ref{eqn:3.7}), (\ref{eqn:3.16}) and (\ref{eqn:3.17}) we obtain
   the Green's formula:
   \begin{multline}\label{eqn:GF}
           \ELEDOT{\Lambda_{1}\big(f\big)}{g}_{\TT{}\times\partial\Omega}
         - \ELEDOT{f}{\Lambda_{2}^{*}\big(g\big)}_{\TT{}\times\partial\Omega}
         = \\
           \sum_{j=0}^{n} r_{j} \Big(
           \ELEDOTT{A_{j}u}{(-i\DXJ{j} v)} 
	  + \ELEDOTT{A_{j}(-i\DXJ{j} u)}{v}
          \Big) \\ 
           + \sum_{j=0}^{n} r_{j}
           \ELEDOTT{\big( (A_{j}^{(2)} )^{2}
          - (A_{j}^{(1)} )^{2}\big)u}{v} 
       - \ELEDOTT{Vu}{v}
   \end{multline}
   where we have set $x_{0}=t$, $A_{j}=A_{j}^{(2)}-A_{j}^{(1)}$ for
   $0\le j\le n$, $V=V^{(2)}-V^{(1)}$, $r_{0}=-1$ and $r_{j}=1$ for $1\le j\le n$.

   At this point it is convenient to notice that if we take the vector and scalar
   potentials in the forward and backward hyperbolic equation to be the same
   $\big($i.e., $(\mathcal{A}^{(1)},V^{(1)})=(\mathcal{A}^{(2)},V^{(2)})\big)$, then we
   get from (\ref{eqn:3.7})
   \begin{equation*}
     \ELEDOT{\Lambda(f)}{g}_{\TT{}\times\partial\Omega}
   - \ELEDOT{f}{\Lambda^{*}(g)}_{\TT{}\times\partial\Omega} = 0
   \end{equation*}
   proving that $\Lambda^{*} = \partial_{\nu} + i \nu\cdot\overline{A(t,x)}$ is the
   $\mathrm{L}^{2}$ adjoint of $\Lambda = \partial_{\nu} + i\nu\cdot A(t,x)$.

   Since we are assuming that the Dirichlet to Neumann maps for the forward and backward
   hyperbolic equations agree ($\Lambda^{1} = \Lambda^{2}$ on
   $\ERRE{}\times\partial\Omega$), then equation (\ref{eqn:GF}) can be rewritten as
   \begin{align}\nonumber
       0  = & \thinspace \ELEDOT{\Lambda_{1}\big(f\big)}{g}_{\TT{}\times\partial\Omega}
           - \ELEDOT{\Lambda_{2}\big(f\big)}{g}_{\TT{}\times\partial\Omega} \\ \nonumber
           = & \ELEDOT{\Lambda_{1}\big(f\big)}{g}_{\TT{}\times\partial\Omega}
           - \ELEDOT{f}{\Lambda_{2}^{*}\big(g\big)}_{\TT{}\times\partial\Omega} \\ \nonumber
          = & \thinspace \sum_{j=0}^{n} r_{j}
      \Big(
         \ELEDOTT{A_{j}u}{ (-i\DXJ{j} v )}
           + \ELEDOTT{A_{j} (-i\DXJ{j} u )}{v}
       \Big) \\ \label{green:for:scalar:pot}
            & + \thinspace \sum_{j=0}^{n} r_{j}
          \ELEDOTT{\big( (A_{j}^{(2)} )^{2}- (A_{j}^{(1)} )^{2}\big)u}{v}
       - \Big\langle Vu,v \Big\rangle_{\TT{}\times\Omega}
   \end{align}
   where as before $x_{0}=t$, $A_{j}=A_{j}^{(2)}-A_{j}^{(1)}$ for $0\le j\le n$,
   $V=V^{(2)}-V^{(1)}$, $r_{0}=-1$ and $r_{j}=1$ for $1\le j \le n$.

%\section{X-ray transform}
   \section{X-ray transform}
  
   Our next step is to combine our two main tools, namely the geometric optics representation
   of the solutions of the forward and backward hyperbolic equations and the Green's formula.

   Owing to (\ref{GO1a}) and (\ref{GO1b}) we can write geometric optics representations for $u$
   and $v$ the solutions of the forward and backward hyperbolic equation respectively. These
   representations are
    \begin{align}\label{geom:optic:1a}
    u(t,x) & = e^{ik(t-\omega\cdot x)- i R_{1}(t,x;\omega)}
             \Big( \chi(t',x') + \mathcal{O}\left(k^{-1}\right) \Big) \\ \label{geom:optic:1b}
    \overline{v(t,x)} & =
               e^{-ik(t-\omega\cdot x)+ i \overline{R_{2}(t,x;\omega)}}
             \Big( \chi(t',x') + \mathcal{O}\left(k^{-1}\right) \Big),
    \end{align}
    where
    \begin{align}\label{geom:op:r1}
    R_{1}(t,x;\omega)
       & = \int^{\frac{1}{2}(t+\omega\cdot x)}_{-\infty}
         \sum_{j=0}^{n} \omega_{j}
         A^{(1)}_{j}(t^{\prime}+s,x^{\prime}+s\omega)
                \thinspace \mathrm{d}s \\ \label{geom:op:r2}
    \overline{R_{2}(t,x;\omega)}
       & = \int^{\frac{1}{2}(t+\omega\cdot x)}_{-\infty}
         \sum_{j=0}^{n} \omega_{j}
         A^{(2)}_{j}(t^{\prime}+s,x^{\prime}+s\omega)
         \thinspace \mathrm{d}s.
    \end{align}

   Notice that for $u$ as in (\ref{geom:optic:1a}), we have for $0\le j\le n$
   \begin{align*}
      \DXJ{j} u
      & = e^{ik(t-\omega\cdot x)- i R_{1}(t,x;\omega)}
             \Big( \DXJ{j}\chi + \mathcal{O}\left(k^{-1}\right)
                    + \left(-ikr_{j}\omega_{j} - i\DXJ{j}R_{1}\right)
                    \left(\chi+\mathcal{O}\left(k^{-1}\right)\right)
             \Big) \\
      & = ke^{ik(t-\omega\cdot x)- i R_{1}(t,x;\omega)}
             \Big(-ir_{j}\omega_{j} \chi + \mathcal{O}(k^{-1})
             \Big),
   \end{align*}
   where we have set $\omega_{0}=1$ and as before $r_{0}=-1$ and $r_{j}=1$ for $1\le j \le n$. Then owing to
   (\ref{geom:optic:1b}) we obtain
   \begin{equation*}
     \left(-i\DXJ{j}u(t,x)\right)\overline{v(t,x)}
     = -ke^{i\left(\overline{R_{2}(t,x;\omega)}-R_{1}(t,x;\omega)\right)}
       \left( r_{j}\omega_{j}\chi(t^{\prime},x^{\prime})^{2} + \mathcal{O}(k^{-1}) \right)
   \end{equation*}
   similarly
   \begin{equation*}
     u(t,x)\overline{\left(-i\DXJ{j}v(t,x)\right)}
     = -ke^{i\left(\overline{R_{2}(t,x;\omega)}-R_{1}(t,x;\omega)\right)}
       \left( r_{j}\omega_{j}\chi(t^{\prime},x^{\prime})^{2} + \mathcal{O}(k^{-1}) \right)
   \end{equation*}

   Thus, Green's formula now reads
   \begin{multline*}
   0=Ck \int_{T_{1}}^{T_{2}}\int_{\Omega} \sum_{j=0}^{n}
       \big( A^{(2)}_{j}(t,x) - A^{(1)}_{j}(t,x)\big) r^{2}_{j}\omega_{j}\chi^{2}(t',x')
             \thickspace \times \\ 
        e^{i\left(\overline{R_{2}(t,x;\omega)}-R_{1}(t,x;\omega)\right)} \thinspace dx\thinspace dt
       + \cdots
   \end{multline*}
   where $C$ is a (negative) constant and ``$\cdots$'' represents terms of order $\mathcal{O}(1)$.
   Dividing the above expression by $Ck$ and taking the limit as $k\to +\infty$ 
   we get
   \begin{equation*}
   0= \int_{T_{1}}^{T_{2}}\int_{\Omega} \sum_{j=0}^{n}\omega_{j}
       \left( A^{(2)}_{j}(t,x) - {A}^{(1)}_{j}(t,x)\right) \chi^{2}(t',x')
              e^{i\left(\overline{R_{2}(t,x;\omega)}-R_{1}(t,x;\omega)\right)} \thinspace dx\thinspace dt.
   \end{equation*}
   Without loss of generality (cf. Remark 3.1 in \cite{Eskin:obstacles}) we can assume 
   that $\mathrm{supp}\thinspace \mathcal{A}^{(j)} \subset \ERRE{}\times\Omega$, $j=1,2$.
   Writing $X^{\prime}=(t^{\prime},x^{\prime})$
   and setting $\mathcal{A} = (A_{0},\dots,A_{n})=\mathcal{A}^{(2)}-\mathcal{A}^{(1)}$
   we get after the change of variables
   $(t,x)= \sigma(1,\omega)+X^{\prime}$
   \begin{equation*}
   0=  \int_{\Pi_{(1,\omega)}}\int_{-\infty}^{\infty}\sum_{j=0}^{n} \omega_{j}
             A_{j}\left(X^{\prime} + \sigma(1,\omega) \right) \chi^{2}\left(X^{\prime}\right)
              e^{i \int_{-\infty}^{\sigma}
                      \sum_{j=0}^{n}\omega_{j} A_{j}\left( X^{\prime} + s(1,\omega) \right)
                   \thinspace \mathrm{d}s
                }
       \thinspace \mathrm{d}\sigma\thinspace \mathrm{d}S_{X^{\prime}}.
   \end{equation*}

   Since $\chi$ is an arbitrary function of $X^{\prime}$ we then conclude that
   \begin{align}\nonumber
   0 & = \int_{-\infty}^{\infty} \sum_{j=0}^{n}\omega_{j}
             A_{j}\left(X^{\prime} + \sigma(1,\omega) \right)
              e^{i \int_{-\infty}^{\sigma}
                      \sum_{j=0}^{n}\omega_{j}A_{j}\left( X^{\prime} + s(1,\omega) \right)
                   \thinspace \mathrm{d}s
                }
       \thinspace \mathrm{d}\sigma \\ \nonumber
     & = -i \int_{-\infty}^{\infty} \frac{\partial}{\partial\sigma}
                \left(e^{i \int_{-\infty}^{\sigma}
                      \sum_{j=0}^{n}\omega_{j}A_{j}\left(X^{\prime} + s(1,\omega) \right)
                    \thinspace \mathrm{d}s}
                \right)
       \thinspace d\sigma \\ \label{eqn:3.18}
     & = -i \left(e^{i \int_{-\infty}^{\infty}
                      \sum_{j=0}^{n}\omega_{j}A_{j}\left(X^{\prime} + s(1,\omega) \right)
                    \thinspace \mathrm{d}s}
            -1 \right)
   \end{align}
   Summarizing we have proven the following

   \begin{lemma}\label{lemma:4.1}
      Suppose that the Dirichlet to Neumann operators $\Lambda_{1}$ and $\Lambda_{2}$
      for the hyperbolic equations
      \begin{equation*}
      L_{k} u = \Big(\big( -i\DT + {A^{(l)}}_{0}(t,x)\big)^{2}
                            - \sum_{j=1}^{n} \big(-i\DXJ{j} + {A^{(l)}}_{j}(t,x) \big)^{2}
                            + V^{(l)}(t,x) \Big) u = 0, \thickspace k=1,2
      \end{equation*}
      equal on $\TT{}\times\partial\Omega$. Then for any light ray
      \begin{equation*}
      \gamma
      =\{ (t',x')+s(1,\omega)\thinspace : s\in\mathbb{R}, \omega\in S^{n-1} ,(t',x')\cdot(1,\omega)=0 \},
      \end{equation*}
      the vectorial ray transform of $\mathcal{A} = (A^{(2)}_{0}-A^{(1)}_{0},\dots,
      A^{(2)}_{n}-A^{(1)}_{n})$ along $\gamma$ is an integer multiple of $2\pi$. This is
      \begin{equation}\label{eqn:3.19}
        \left(\mathcal{P}\mathcal{A}\right)(t,x;\omega) : =
        \int_{-\infty}^{\infty}
            \sum_{j=0}^{n}\omega_{j}A_{j}(t'+s,x'+s\omega)\thinspace
        \mathrm{d}s = 2\pi r
      \end{equation}
      for some $r\in\mathbb{Z}$. Here $(t',x') = (t,x) -
      \frac{1}{2}(t+\omega\cdot x)(1,\omega)$ and $\omega_{0}=1$.
   \end{lemma}
   \begin{proof}
   Equation (\ref{eqn:3.18}) can be rewritten as
   \begin{equation*}
      e^{i\left(\mathcal{P}\mathcal{A}\right)(t,x;\omega)} = 1
   \end{equation*}
   which in turn implies (\ref{eqn:3.19}).
   \end{proof}

   If we now incorporate the hypothesis  of $\mathcal{A}^{(1)}$ and
   $\mathcal{A}^{(2)}$ being compactly supported in $x$ we can determine
   the exact value of $r$. This is because equation (\ref{eqn:3.19})
   holds for any $(t,x;\omega)\in \mathbb{R}_{t,x}^{m+1}\times S^{m-1}$
   and in particular, when $t=0$ and $|x|$ is big enough and perpendicular
   to a fixed $\omega$, the light ray $(0,x)+s(1,\omega)$, $s\in\mathbb{R}$
   does not meet the support of $\mathcal{A}$, hence
    \begin{equation}\label{eqn:3.20}
    \int_{-\infty}^{\infty}
    \sum_{j=0}^{m}\omega_{j}A_{j}(t^{\prime}+s,x^{\prime}+s\omega) = 0.
    \end{equation}

   To conclude this section let us proceed to remove the condition
   $(t^{\prime},x^{\prime})\cdot(1,\omega)=0$. If $(t,x)$ is an arbitrary
   point in $\mathbb{R}^{n+1}_{t,x}$, then by making the change
   of variables $s=\sigma - \frac{t+\omega\cdot x}{2}$ we get
   \begin{equation*}
   \int_{-\infty}^{\infty} \sum_{j=0}^{n}\omega_{j}A_{j}(t+\sigma,x+\sigma\omega)
        \thinspace\mathrm{d}\sigma =
   \int_{-\infty}^{\infty} \sum_{j=0}^{n}\omega_{j}A_{j}(t^{\prime}+s,x^{\prime}+s\omega)
        \thinspace\mathrm{d}s,
   \end{equation*}
   where $(t^{\prime},x^{\prime})= (t,x) - \frac{(t+\omega\cdot x)}{2}(1,\omega)$.
   Clearly this last integral equals zero by (\ref{eqn:3.20}).

%\section{The main theorem}\label{sec:proofs:part1}
   \section{The main theorem}\label{sec:proofs:part1}

   In this section we will establish several uniqueness results for
   vector and scalar potentials satisfying different growth conditions.
   Let us proceed first with the part that deals with the vector potentials
   in the case when the component functions $A_{j}(t,x)$ decay exponentially
   in $t$.

    \begin{thm}\label{prop:uniq}
       Suppose that $\mathcal{A}(t,x) =
       \big(A_{0}(t,x),\dots,A_{n}(t,x) \big)$ 
       with $\mathcal{A} \in C^{\infty}$ in $x$ and $t$ is such that for
       any non-negative integers $\alpha$, $\beta$ and for any $ 0 \le j \le n $  
       there exist positive constants $c$, $C_{\alpha,\beta}$ such that 
       for $|t| \ge t_{0}$,
       $| \partial^{\alpha}_{t} \partial^{\beta}_{x} A_{j}(t,x) | \le 
       C_{\alpha,\beta}e^{-c |t|}$. 
       If in addition $A_{j}(t,x) = 0$ for
       $|x| \ge R > 0$ and (\ref{eqn:3.20}) holds, then there exists 
       $\varphi(t,x)\in C^{\infty}(t,x)$ and positive constants 
       $C^{\prime}_{\alpha,\beta},c'$ such that
       \begin{itemize}
       %(
       \item[\textit{i)}] $A_{0}(t,x) = \DT \varphi(t,x),\quad$ $A_{j}(t,x)
                                      = \DXJ{j}\varphi(t,x),\quad$
                          $1\le j \le n$, and
       \item[\textit{ii)}] $\mathrm{Supp}\thinspace\varphi \subseteq
                            \mathbb{R} \times \{ |x| \le R \},\quad$
			    $|\partial^{\alpha}_{t}\partial^{\beta}_{x}\varphi(t,x)| 
			     \le C^{\prime}_{\alpha,\beta}e^{-c'|t|}$.
       \end{itemize}
    \end{thm}
    \begin{proof}
    By uniqueness of the Fourier transform and by (\ref{eqn:3.20}) we have
    \begin{equation*}
    0 = \iint e^{-it\tau-ix\cdot\xi}\Big(\int_{-\infty}^{\infty}
    \sum_{j=0}^{n}\omega_{j}A_{j}(t+s,x+s\omega)\thinspace \mathrm{d}s\Big)
    \thinspace dt dx.
    \end{equation*}

    By the hypothesis on the support of $A_{j}$, $1\le j \le n$,
    we can change the order of integration.  After writing
    $t_{1} = t + s$, $x_{1} = x + s\omega$, we are lead to
    \begin{align}\nonumber
    0 & = \iiint e^{-i(t_{1}-s)\tau-i(x_{1}-s\omega)\cdot\xi}
                 \Big(A_{0}+\sum_{j=1}^{n}\omega_{j}A_{j}\Big)(t_{1},x_{1})
          \thinspace dt_{1}dx_{1}ds \\ \nonumber
      & = \int e^{is(\tau+\omega\cdot\xi)}
                 \Big(A_{0}+\sum_{j=1}^{n}\omega_{j}A_{j}\Big)^{\wedge}(\tau,\xi)
          \thinspace ds \\ \label{eqn:3.21}
      & = \delta(\tau+\omega\cdot\xi)\Big(A_{0}+\sum_{j=1}^{n}\omega_{j}A_{j}\Big)^{\wedge}(\tau,\xi),
    \end{align}
    which tells us that the Fourier transform of $A_{0}+\sum_{j=1}^{n}\omega_{j}A_{j}$
    vanishes on $\Pi_{(1,\omega)}$, the hyperplane perpendicular to
    $(1,\omega)$.

    We claim that this Fourier transform vanishes in the
    complement of the `light cone' $\mathcal{C} =
    \{(\tau,\xi) \thinspace : \thinspace |\tau|\ge|\xi|\}$
    for an appropriate choice of $\omega$.
    To see this notice that if $(\tau,\xi)\not\in\mathcal{C}$,
    then $\frac{-\tau}{|\xi|}$ has norm less than one and we can
    find $\omega = \omega(\tau,\xi) \in S^{n-1}$ satisfying
    \begin{equation}\label{eqn:3.22}
       \frac{\xi}{|\xi|}\cdot\omega(\tau,\xi) = -\frac{\tau}{|\xi|}.
    \end{equation}

    Clearly with this choice of $\omega$ we have
    $\tau+\omega(\tau,\xi)\cdot\xi=0$
    and the function
    $\big(A_{0}+\sum_{j=1}^{n}\omega_{j}(\tau,\xi)A_{j}\big)^{\wedge}(\tau,\xi)$
    vanishes when $|\tau| < |\xi|$. Morever, this shows that the Fourier
    transform of the vector potential $\widehat{\mathcal{A}}(\tau,\xi)$ is
    perpendicular to the $(n+1)$-dimensional vector $(1,\omega(\tau,\xi))$
    as
    \begin{equation}\label{eqn:3.22.2}
    \big(A_{0}+\sum_{j=1}^{n}\omega_{j}(\tau,\xi)A_{j}\big)^{\wedge}(\tau,\xi)
    = (1,\omega(\tau,\xi)) \cdot \widehat{\mathcal{A}}(\tau,\xi).
    \end{equation}
    Equation (\ref{eqn:3.22}) has infinitely many solutions and as a matter of
    fact they can be parametrized by $S^{n-2}$.
    On the other hand, equation (\ref{eqn:3.22.2}) tells us that
    $\widehat{\mathcal{A}} = (\widehat{A}_{0},\dots,\widehat{A}_{n}) $ is
    orthogonal to all elements of $E=\{(1,\omega(\tau,\xi)) :
    \tau + \omega(\tau,\xi)\cdot\xi =0 \}$.
    It is not hard to prove (see Appendix A) that the orthogonal complement
    $E^{\perp}$ is one dimensional and since $(\tau,\xi)$ is perpendicular
    to any vector of the form $(1,\omega(\tau,\xi))$, this complement has
    to agree with the line $\{ c(\tau,\xi)\thinspace :\thinspace c\in\ERRE{}\}$.

    Since the previous argument works for an arbitrary $\tau$ and since the set
    $\{ \xi\thinspace : \thinspace |\tau| < |\xi| \}$ is an open subset in
    $\ERRE{n}$, we see that
    $\widehat{\mathcal{A}}(\tau,\xi) =
    \big(\widehat{A_{0}}(\tau,\xi),\dots,\widehat{A_{n}}(\tau,\xi)\big)$
    is proportional to the vector $(\tau,\xi)$ in the complement of the
    light cone. In other words, we can find a function $\Phi$ such that
    \begin{equation} \label{eq:xi:analytic}
      (\widehat{A_{0}}(\tau,\xi),\dots,\widehat{A_{n}}(\tau,\xi))
      =i\Phi(\tau,\xi)\thinspace(\tau,\xi)
    \end{equation}
    whenever $|\tau|<|\xi|$. Since for any $j$ the function $A_{j}$ decays 
    exponentially in $t$ and is compactly supported in $x$ then its Fourier 
    transform $\widehat{A_{j}}$ is analytic in the strip 
    $|\mathrm{Im}\thickspace\tau| < c$.
    
    On the other hand, equation (\ref{eq:xi:analytic}) gives   
    \begin{align*}
    \Phi(\tau,\xi) & = -\frac{i\widehat{A_{j}}(\tau,\xi)}{\xi^{(j)}},\quad{1\le j\le n}, \\
    \Phi(\tau,\xi) & = -\frac{i\widehat{A_{0}}(\tau,\xi)}{\tau},
    \end{align*}
    which tell us that
    $\Phi$ is analytic in the set $\{(\tau,\xi)\thinspace:\thinspace 
    |\mathrm{Im}\thickspace\tau| < c,\thinspace(\tau,\xi)\neq(0,0)\}$.
    Hartog's theorem (see \cite{Hartog}) tells us that the
    concepts of removable singularities and isolated singularities agree in
    functions of several complex variables and we conclude that
    $\Phi$ is analytic in the strip $|\mathrm{Im}\thickspace\tau|<c$. 
    Moreover if we let $\varphi$ be the inverse Fourier transform of
    $\Phi$, then $\varphi$ and all of its derivatives are exponentially 
    decaying in $t$ and we only need to make sure that it has the right 
    support properties.

    Because of the assumptions on the support of the functions $A_{j}$ we have,
    by the Paley-Wiener theorem
    \begin{equation*}
    \big| \widehat{A_{j}}(\tau,\xi)\big| \le
          \frac{C_{N}^{(j)}\exp(R|\mathrm{Im}\thinspace\xi|)}{(1+|\xi|)^{N}},\qquad 0\le j\le n,
    \end{equation*}
    for $1\le j \le n$. Then for $|\xi^{(j)}| > 1$
    \begin{equation*}
    \big|\Phi(\tau,\xi)\big| =
    \left|\frac{\widehat{A_{j}}(\tau,\xi)}{\xi^{(j)}}\right| \le
          \frac{C_{N}^{(j)}\exp(R|\mathrm{Im}\thinspace\xi|)}{(1+|\xi|)^{N}},\quad 0\le j\le n,
    \end{equation*}
    for some $C_{N}>0$. Since the function
    %When $|\xi^{(j)}|\le 1$, the function 
    $h(\tau,\xi) =
    \Phi(\tau,\xi)(1+|\xi|)^{N}\exp(-R|\mathrm{Im}\thinspace\xi|)$ is continuous
    when $|\xi^{(j)}|\le 1$, it is also 
    %hence 
    bounded, hence $|h(\tau,\xi)| \le C_{N}$ for some positive $C_{N}$ and
    the estimate
    \begin{equation}\label{Paley-Wiener}
    \big|\Phi(\tau,\xi)\big| \le
          \frac{C\exp(R|\mathrm{Im}\thinspace\xi|)}{(1+|\xi|)^{N}}
    \end{equation}
    holds for any $\xi\in\ERRE{n}$. Making use once again of the
    Paley-Wiener theorem we conclude that the inverse fourier transform 
    of $\Phi(\tau,\xi)$ is supported in the set $\{x : |x| \le R\}$.
    \end{proof}

    Before going ahead to prove the correspoding equality of the scalar potentials,
    we will pause for a second to relax the conditions imposed on the vector potential.
    Let us start by replacing exponentially decaying by Schwartz functions.
    \begin{thm}
       Suppose that $\mathcal{A}(t,x) =
       \big(A_{0}(t,x),\dots,A_{n}(t,x) \big)$ 
       with $\mathcal{A}\in C^{\infty}$ in $x$ and $t$ is such 
       that for any $M>0$ and non-negative integers $\alpha$,$\beta$ 
       there exist constants $C_{M,\alpha,\beta} > 0$ such that
       $(1+|t|)^M |\partial^{\alpha}_{t} \partial^{\beta}_{x}A_{j}(t,x) | 
       \le C_{M,\alpha,\beta}$
       for $0\le j \le n$.
       If in addition $A_{j}(t,x) = 0$ for
       $|x| \ge R > 0$ and (\ref{eqn:3.20}) holds, then there exists 
       $\varphi(t,x)\in C^{\infty}(t,x)$ such that
       \begin{itemize}
       %(
       \item[\textit{i)}] $A_{0}(t,x) = \DT \varphi(t,x),\quad$ $A_{j}(t,x)
                                      = \DXJ{j}\varphi(t,x),\quad$
                          $1\le j \le n$, and
       \item[\textit{ii)}] $\mathrm{Supp}\thinspace\varphi \subseteq
                            \mathbb{R} \times \{ |x| \le R \},\quad$
                            $(1+|t|)^M |\partial^{\alpha}_{t} \partial^{\beta}_{x}
			     \varphi(t,x) | \le C^{\prime}_{M,\alpha,\beta}$. 
       \end{itemize}
    \end{thm}
    \begin{proof}
    The proof goes along the same lines as the previous proposition except that 
    now in equation (\ref{eq:xi:analytic}) we only know that the left hand side 
    is entire in $\xi$. For $\tau_{0}\neq 0$ fixed, Hartog's theorem tells us 
    that $\Phi(\tau_{0},\xi)$ is entire and when $\tau = 0$, equation
    (\ref{eq:xi:analytic}) gives 
    $\Phi(0,\xi) = -\dfrac{i\widehat{A_{j}}(0,\xi)}{\xi^{(j)}}$
    for $1\le j \le n$ showing that $\Phi$ has no singularities.
     
    The part of the proof that deals with the support of $\varphi$ remains unchanged and
    we only need to show that $\Phi$ is a Schwartz function.
    If $M^{\prime}>0$ and $\beta$ is a non-negative integer, we have for $|\xi|\le R$
    \begin{align*}
           (1+|\tau|)^{\tilde{M}} \left| \partial^{\beta}_{\tau} \Phi(\tau,\xi)\right|
       & = (1+|\tau|)^{\tilde{M}} \left| \partial^{\beta}_{\tau} \left( i\frac{\widehat{A_{0}}(\tau,\xi)}{\tau}
           \right) \right| \\
       & = (1+|\tau|)^{\tilde{M}} \left| \sum_{j=0}^{\beta} c_{j}
           \partial^{j}_{\tau} \widehat{A_{0}}(\tau,\xi) \partial_{\tau}^{\beta-j} \left( \frac{1}{\tau}
       \right) \right|   \\
       & \le C (1+|\tau|)^{\tilde{M}^{\prime}} \left| \sum_{j=0}^{\beta}
           \partial^{j}_{\tau} \widehat{A_{0}}(\tau,\xi) \right| \le C_{\tilde{M}^{\prime},\beta,R},
    \end{align*}
    where in the last inequality we used the fact that the exponentially decaying
    $C^{\infty}$ functions Fourier transform into Schwartz functions. Since $\Phi$ is itself
    Schwartz, the desired function $\varphi$ is again the inverse Fourier
    transform of $\Phi$.
    \end{proof}
    The conditions on the vector potential imposed so far are such that
    we end up working with functions once we compute the Fourier transform of equation
    (\ref{eqn:3.20}), nevertheless, this transform can be computed under weaker assumptions
    and the following theorem tells us that the final result is still valid.\\

    \begin{thm}\label{prop:uniq2}
       Suppose that $\mathcal{A}(t,x) =
       \big(A_{0}(t,x),\dots,A_{n}(t,x) \big)$ with 
       $\mathcal{A} \in C^{\infty}$ in $x$ and $t$ is such that for
       $ 0 \le j \le n $,
       $| A_{j}(t,x) | \le C(1+ |t|)^{M}$ with $C,M>0$ and $|t| \ge t_{0}$.
       If in addition the functions $|A_{j}(t,x)|$ are locally integrable in
       $\ERRE{n+1}$, satisfy the support condition $A_{j}(t,x) = 0$
       for $|x| \ge R > 0$, and equation (\ref{eqn:3.20}) holds;
       then there exists $\varphi(t,x)\in C^{\infty}(t,x)$ such that
       \begin{itemize}
       %(
       \item[\textit{i)}] $A_{0}(t,x) = \DT \varphi(t,x),\quad$ $A_{j}(t,x)
                                      = \DXJ{j}\varphi(t,x),\quad$
                          $1\le j \le n$, and
       \item[\textit{ii)}] $\mathrm{Supp}\thinspace\varphi \subseteq
                            \mathbb{R} \times \{ |x| \le R \}$.
       \end{itemize}
    \end{thm}
    \begin{proof}
       By the hypothesis on the growth of $A_{j}$, $1\le j \le  n$, we can compute the
       Fourier transform of equation (\ref{eqn:3.20}) to obtain
       $ \delta(\tau + \omega\cdot\xi) \big( A_{0} + \sum_{j=1}^{n} \omega_{j}A_{j} \big)^{\wedge}
       (\tau,\xi) =0$, where $\widehat{A}_{j}(\tau,\xi)$ is an analytic function in $\xi$ and a
       distribution in $\tau$. In addition, since the wavefront set of $\delta(\tau + \omega\cdot\xi)$
       and $\widehat{A}_{0} + \sum_{j=1}^{n}\omega_{j}\cdot\widehat{A}_{j}$ do not intersect,
       we can define a restriction of $\widehat{A}_{0} + \sum_{j=1}^{n}\omega_{j}\cdot\widehat{A}_{j}$
       on the hyperplane $\tau + \omega\cdot\xi = 0$ (cf. H\"ormander \cite{Hormander}). Proceeding 
       as before we find that when
       $|\tau| < |\xi|$ there are infinitely many solutions of equation
       (\ref{eqn:3.22}) and that they can be parametrized by $S^{n-2}$. Moreover, the change
       $(\tau,\xi) \to (\alpha\tau,\alpha\xi)$, $\alpha >0$,
       in (\ref{eqn:3.22}) leads to
       \begin{align*}
          \frac{\alpha \xi}{ | \alpha | |\xi|} \cdot \omega(\alpha\tau,\alpha\xi)
     & = - \frac{\alpha\tau}{|\alpha| |\tau|} \\
          \frac{\xi}{ |\xi|} \cdot \omega(\alpha\tau,\alpha\xi)
     & = - \frac{\tau}{|\tau|},
       \end{align*}
       which tells us that the solutions $\omega(\tau,\xi)$ of (\ref{eqn:3.22}) are homogeneous
       of degree 0 in $(\tau,\xi)$.

       Therefore 
      \begin{equation}\label{eqn:vec:pot:with:alpha:non:homo}
          \widehat{A}_{0}(\tau,\xi) +
         \sum_{j=1}^{n} \omega_{j}\widehat{A}_{j}(\tau,\xi)= 0
      \end{equation}
      on the plane $\tau + \omega\cdot\xi = 0$.
%
      %One more time applying the arguments employed in proposition \ref{prop:uniq} about the
      %orthogonal complement
      %of the set $\{ (1,\omega(\tau,\xi)) \thickspace : \thickspace \tau +\omega(\tau,\xi)\cdot\xi = 0,
      %\thinspace |\tau| < |\xi| \}$ 
      %we conclude that
      %the vectors $(\widehat{A}_{0}(\tau,\xi),\widehat{A}_{1}(\tau,\xi),\dots,\widehat{A}_{n}(\tau,\xi))$
      %and $(\tau,\xi_{1},\dots,\tau_{n})$ are colinear. In symbols
      %\begin{equation*}
      %   (\widehat{A}_{0}(\tau,\xi),\dots,\widehat{A}_{n}(\tau,\xi))
      %= i \Phi(\tau,\xi) (\tau,\xi)
      %\end{equation*}
      %for some $\Phi(\tau,\xi)$ whenever $|\tau| < |\xi|$.
%
      %We point out that the main difficulty here lies in the fact that
      %we are now dealing with distributions in $\tau$, however we still have that
      %$ \widehat{A}_{0}(\tau,\xi) + \sum_{j=1}^{n} \omega_{j}(\tau,\xi)\widehat{A}_{j}(\tau,\xi)= 0$
      %for $|\tau| < |\xi|$. Using the homogeneity of $\omega$ we can 

      Replace
      $(\tau,\xi)$ by $(\alpha\tau,\alpha\xi)$ where $\alpha>0$ and 
      %still have $\omega(\alpha\tau,\alpha\xi)=\omega(\tau,\xi)$, 
      we then see that for
      $\alpha>0$
      \begin{equation}\label{eqn:vec:pot:with:alpha}
          \widehat{A}_{0}(\alpha\tau,\alpha\xi) +
         \sum_{j=1}^{n} \omega_{j}\widehat{A}_{j}(\alpha\tau,\alpha\xi)= 0.
      \end{equation}
      We now let $\chi(\alpha)$ be an arbitrary $C_{0}^{\infty}(\ERRE{})$ function
      with support contained in the set $|\alpha - 1 |< \epsilon$ and multiply
      (\ref{eqn:vec:pot:with:alpha}) by $\chi(\alpha)$. Integration in
      $\alpha$ leads to
      \begin{equation}\label{eq:52}
          a_{0}(\tau,\xi) +
         \sum_{j=1}^{n} \omega_{j}a_{j}(\tau,\xi)= 0,
      \end{equation}
      where $\tau + \omega\cdot\xi = 0 $ and
      \begin{equation*}
          a_{j}(\tau,\xi) = \int_{-\infty}^{\infty}
         \widehat{A}_{j}(\alpha\tau,\alpha\xi) \chi(\alpha)
         \thinspace d\alpha.
      \end{equation*}
      Notice that $a_{j}(\tau,\xi)$ are no longer distributions and we can put
      $\omega  = \omega(\tau,\xi)$ in (\ref{eq:52}).
      Arguing as before we find that
      $\big(a_{0}(\tau,\xi),\dots,a_{n}(\tau,\xi)\big)=ib(\tau,\xi)(\tau,\xi)$
      for some $b(\tau,\xi)$, or in other words,
      \begin{equation*}
         \frac{a_{0}(\tau,\xi)}{\tau}  =
         \frac{a_{1}(\tau,\xi)}{\xi_{1}}  =
     \cdots = \frac{a_{n}(\tau,\xi)}{\xi_{n}}  =
     ib(\tau,\xi).
      \end{equation*}
      Since $\chi(\alpha)$ is arbitrary we get
      \begin{equation*}
         \frac{\widehat{A}_{0}(\alpha\tau,\alpha\xi)}{\alpha\tau}  =
         \frac{\widehat{A}_{1}(\alpha\tau,\alpha\xi)}{\alpha\xi_{1}}  =
     \cdots = \frac{\widehat{A}_{n}(\alpha\tau,\alpha\xi)}{\alpha\xi_{n}}  =
     i\widehat{\Psi}(\alpha\tau,\alpha\xi),
      \end{equation*}
      where $\widehat{\Psi}(\alpha\tau,\alpha\xi)$ is a distribution in $\alpha\tau$ for
      all $\alpha\in(1-\epsilon,1+\epsilon)$. Finally when $\alpha = 1$ we get
      \begin{equation*}
         \widehat{A}_{0}(\tau,\xi)=i\tau\widehat{\Psi}(\tau,\xi), \quad
         \widehat{A}_{1}(\tau,\xi)=i\xi_{1}\widehat{\Psi}(\tau,\xi),
     \quad\dots\quad,
         \widehat{A}_{n}(\tau,\xi)=i\xi_{n}\widehat{\Psi}(\tau,\xi)
      \end{equation*}
      for $|\tau| < |\xi| $. As before we have that $\widehat{\Psi}$ is entire in
      $\xi$ and $\widehat{\Psi}\in S^{\prime}$ in $\tau$ (since $A_{j}\in S^{\prime}$
      in $\tau$). Therefore $\varphi = \mathcal{F}^{-1}_{\tau,\xi}\widehat{\Psi} \in S^{\prime}$
      in $t$. Moreover the identies $ \partial_{t}\varphi = A_{0},
         \partial_{x_{1}}\varphi = A_{1}, \dots \partial_{x_{n}}\varphi = A_{n} $ 
      imply that $ \varphi(t,x)\in C^{\infty}$ in $(t,x)$ and that 
      $\varphi = 0$ for $|x|>R$.
    \end{proof}

    Summarizing, the three previous results prove that the vector potentials
    $\mathcal{A}^{(1)}$ and $\mathcal{A}^{(2)}$ are gauge equivalent with
    gauge $g=e^{i\varphi}$.
    Next, we show that this equivalence implies the equality of the scalar
    potentials $V^{(1)}$ and $V^{(2)}$.

    By the previous proposition
    $\mathcal{A}^{(2)}-\mathcal{A}^{(1)} = \nabla_{t,x}\varphi$, replacing the pair
    $\big(\mathcal{A}^{(1)},V^{(1)}\big)$ by $\big(\mathcal{A}^{(3)},V^{(3)}\big)$ where
    $\mathcal{A}^{(3)} =\mathcal{A}^{(1)} + \nabla_{t,x}\varphi $ and
    $V^{(1)}=V^{(3)}$, we find by means of proposition \ref{prop1} that
    $\mathcal{A}^{(3)}=\mathcal{A}^{(2)}$. Next we use our Green's formula
    (\ref{green:for:scalar:pot}) with the pair of potentials $\big(\mathcal{A}^{(2)},V^{(2)}\big)$
    and $\big(\mathcal{A}^{(3)},V^{(3)}\big)$ to obtain
    \begin{equation*}
       0 =
       \ELEDOTT{\big(V^{(3)}-V^{(2)}\big)u}{v}
       = \int_{T_{1}}^{T_{2}} \int_{\Omega}
         \left(V^{(3)}-V^{(2)}\right) u \overline{v}\thinspace dx dt.
    \end{equation*}
    Making use of the geometric optics representations (\ref{geom:optic:1a})-(\ref{geom:op:r2})
    the above integral becomes
    \begin{equation}\label{eqn:for:V}
    0=\int_{T_{1}}^{T_{2}} \int_{\Omega}
          \left(V^{(3)}(t,x)-V^{(2)}(t,x)\right)
                e^{i\left(\overline{R_{2}(t,x;\omega)}-R_{1}(t,x;\omega)\right)}
                \chi^{2}(t',x')
      \thinspace dx dt + \cdots \thinspace
    \end{equation}
    where $\cdots$ denotes terms of order $\mathcal{O}\left(k^{-1}\right)$.
    Taking the limit as $k\to +\infty$ we
    notice that the equality of the vector potentials imply that
    $\overline{R_{2}} - R_{1} = 0$. Thus, after a change of variables, we can
    rewrite (\ref{eqn:for:V}) as
    \begin{equation*}
    0=\int_{\Pi_{(1,\omega)}} \left( \int_{-\infty}^{\infty}
          V^{(3)}\big(X^{\prime} + s(1,\omega)\big)-
      V^{(2)}\big(X^{\prime} + s(1,\omega)\big) \thinspace
      \mathrm{d}s \right)
                \chi^{2}\left(X^{\prime}\right)
      \thinspace \mathrm{d}S_{X^{\prime}},
    \end{equation*}
    since $\chi$ is arbitrary the inner integral in the expression above vanishes
    \begin{equation}
      \int_{\infty}^{\infty}
      \left(
      V^{(3)}(t^{\prime}+s,x^{\prime}+s\omega) -
      V^{(2)}(t^{\prime}+s,x^{\prime}+s\omega)
      \right)
      \thinspace \mathrm{d}s = 0
    \end{equation}
    which shows that the light ray transform of the potentials agree.
    A simple variation of the previous proof applies 
    and we have $V^{(1)}=V^{(3)}=V^{(2)}$. Therefore
    the pair of potentials $\left(\mathcal{A}^{(1)},V^{(1)}\right)$ and
    $\left(\mathcal{A}^{(2)},V^{(2)}\right)$ are gauge equivalent (see also
    \cite{Stefanov},\cite{Ramm:Sjostrand}).

%\section{Stability Estimate}\label{sec:stability:estimates}
\section{Stability estimate}\label{sec:stability:estimates}

    In this section we assume that the components of the vector potentials
    $\mathcal{A}^{(1)}$ and $\mathcal{A}^{(2)}$ are real valued, smooth and 
    compactly supported in both $t$ and $x$. Just as we did before, let us write
    \begin{equation*}
       \mathcal{A} = \mathcal{A}^{(1)} - \mathcal{A}^{(2)} \quad\text{where}\quad
       \mathcal{A}^{(k)} = (A_{0}^{(k)},\dots,A_{n}^{(k)}),\quad k=1,2,
    \end{equation*}
    and let us further assume that the potential $\mathcal{A}$ satisfies the
    divergence condition
    \begin{equation}\label{condition:divergence}
       \mathrm{div}\thinspace\mathcal{A} = 
       \partial_{t}A_{0}(t,x) + \sum_{j=1}^{n} \partial_{x_{j}}A_{j}(t,x) = 0.
    \end{equation}
    Since the potentials are compactly supported we can find real numbers 
    $T_{1} < 0 < T_{2}$ and an open set $\mathcal{D} \subseteq \ERRE{}_{t}\times\ERRE{n}_{x}$
    such that if we set $Q=(T_{1},T_{2})\times\Omega$, then
    \begin{itemize}
       \item[\textit{i)}] $Q\subseteq\mathcal{D}$; and
       \item[\textit{ii)}] $\mathrm{Supp}(A_{j})\subseteq\mathcal{D}$ for $j=0,1,\dots,n$.
    \end{itemize}
    In other words, $\mathcal{D}$ is a set containing $Q$ and the support of all the
    components of the vector potential. Since $\mathcal{D}$ is bounded we can find 
    $T_{3}$ bigger than $|T_{1}|$ and $|T_{2}|$ such that $\mathcal{D} 
    \subseteq (-T,T)\times\ERRE{n}_{x}$, we then choose $T>T_{3}+\mathrm{diam}(\Omega)$
    and set $Q_{T}:=(-T,T)\times\Omega$.

     When $T$ is selected appropriately we can find solutions
     \begin{equation*}
        u(t,x) = \chi_{\omega}(t,x)
              e^{ik(t-\omega\cdot x) + i R_{1}(t,x;\omega) }
         \left( 1 + \mathcal{O} (k^{-1})\right)
     \end{equation*}
     and
     \begin{equation*}
        v(t,x) = \chi_{\omega}(t,x)
              e^{ik(t-\omega\cdot x) + i \overline{ R_{2}(t,x;\omega) } }
         \left( 1 + \mathcal{O} (k^{-1})\right)
     \end{equation*}
     of the backward and forward hyperbolic equations satisfying
     \begin{align*}
        u &= u_{t} = 0\quad\text{ on }\quad\{-T\}\times\Omega, \\
        v &= v_{t} = 0\quad\text{ on }\quad\{T\}\times\Omega,
     \end{align*}
     where the function $\chi_{\omega}$ is such that
    \begin{itemize}
       \item[a)] $(\partial_{t} + \sum_{j=1}^{n}\omega_{j}\partial_{x_{j}}) \chi_{\omega}(t,x) = 0$.
                 This is, $\chi_{\omega}$ is constant along light rays; and
       \item[b)] $\chi_{\omega}$ is supported in a small neighborhood of the ray 
		 $\{ (t+s,x+s\omega) \thickspace : \thickspace s\in\ERRE{} \}$.
    \end{itemize}

    As before we can make use of the Green's formula developed in previous sections to obtain
    \begin{align} \nonumber
      \ELEDOT{(\Lambda_{1} - \Lambda_{2})(f)}{g}_{(-T,T)\times\partial\Omega} = I_{Q_{T}}
          &  := \sum_{j=1}^{n} \big(
           \ELEDOTTT{A_{j}u}{ -i\DXJ{j} v }
         + \ELEDOTTT{A_{j}(-i\DXJ{j} u)}{v}
          \big) \\ \nonumber
          & \thinspace + \sum_{j=1}^{n} 
           \ELEDOTTT{\big( (A_{j}^{(2)} )^{2}
          - (A_{j}^{(1)} )^{2}\big)u}{v} 
          - \ELEDOTTT{Vu}{v} \\ \nonumber
         & - \ELEDOTTT{A_{0}u}{ -i\partial_{t} v }
         - \ELEDOTTT{A_{0}(-i\partial_{t} u)}{v} \\ \nonumber
         & - \ELEDOTTT{\big( (A_{0}^{(2)} )^{2}
            - (A_{0}^{(1)} )^{2}\big)u}{v}  \\ \label{int:QT}
    \end{align}
    where we have set
    \begin{equation*}
       f = u(t,x) \big|_{(-T,T)\times\partial\Omega} \qquad
       g = v(t,x) \big|_{(-T,T)\times\partial\Omega}.
    \end{equation*}

    We next regard $\Lambda_{1}- \Lambda_{2}$ as a map from 
    $H^{1}\to H^{0}$ and denoting by $|||\phantom{m}|||$ the operator norm 
    between these spaces, we have by the Cauchy-Schwarz inequality
    %for a function $h \in
    %$H^{s}\to
    %H^{s-1}$ and denoting by $|||\phantom{m}|||$ the operator norm 
    %between these spaces, we have for a function $h \in
    %H^{s}((-T,T)\times\partial\Omega)$
    %\begin{equation*}
    %   \big| \big| (\Lambda_{1} - \Lambda_{2}) (h) \big|\big|_{H^{s-1}(Q_{T})} \le
    %   ||| \Lambda_{1} - \Lambda_{2} |||\thinspace\thinspace || h ||_{H^{s}(Q_{T})}.
    %\end{equation*}
    %In the particular case when $s=1$ we obtain from (\ref{int:QT}) and the
    %Cauchy-Schwarz inequality
    %\begin{equation*}
    \begin{multline*}
       | I_{Q_{T}} | =
       \big| \ELEDOT{(\Lambda_{1} - \Lambda_{2})(f)}{g}_{(-T,T)\times\partial\Omega} \big| \le
       ||| \Lambda_{1} - \Lambda_{2} ||| \times \\
       \thinspace\thinspace || f ||_{H^{1}((-T,T)\times\partial\Omega)} 
                 \thinspace || g ||_{L^{2}((-T,T)\times\partial\Omega)}.
    \end{multline*}
    %\end{equation*}
    The latter norm can be estimated by
    \begin{align}\nonumber 
       || \thinspace g \thinspace ||_{L^{2}((-T,T)\times\partial\Omega)}
       & =   || \thinspace \chi_{\omega}(t,x) 
              (1+\mathcal{O}(k^{-1})) \thinspace ||_{L^{2}((-T,T)\times\partial\Omega)} \\ \label{norm:g:L2}
       & \le || \thinspace \chi_{\omega}(t,x) 
       \thinspace ||_{L^{2}((-T,T)\times\partial\Omega)}
              + \mathcal{O}(k^{-1}),
    \end{align}
    whereas the middle norm can be estimated by
    \begin{align}\nonumber
       || \thinspace f \thinspace ||_{H^{1}((-T,T)\times\partial\Omega)}
       & \le C(n,\Omega) \Big[
           k \big|\big|\thinspace \chi_{\omega} 
                \thinspace\big|\big|_{L^{2}((-T,T)\times\partial\Omega)}
         + \mathcal{O}(1) \Big]  \\ \label{norm:f:H1}
       & = C(n,\Omega)k \Big[
           \big|\big|\thinspace \chi_{\omega} 
                \thinspace\big|\big|_{L^{2}((-T,T)\times\partial\Omega)}
           + \mathcal{O}(k^{-1}) \Big].
    \end{align}
    If in addition we assume that 
    $ ||\chi_{\omega}||_{L^{2}((-T,T)\times\partial\Omega)} \le C $ 
    we have by (\ref{norm:g:L2}) and (\ref{norm:f:H1}) 
    \begin{equation} \label{stab:estim:1}
       | I_{Q_{T}} | \le
       kC(n,\Omega) \Big[ |||\Lambda_{1}-\Lambda_{2} ||| 
       + \mathcal{O}(k^{-1}) \Big],
    \end{equation}
    where $C(n,\Omega)$ is a constant that depends upon the supremum
    of all ray integrals of $A_{0} + \sum_{j=1}^{n} \omega_{j}A_{j}$.
    On the other hand, the integral $I_{Q_{T}}$ leads, just as before, to
    \begin{multline*}
       \Bigg| Ck \int_{-T}^{T}\int_{\Omega} 
          \Big( A_{0} + \sum_{j=1}^{n} \omega_{j}A_{j}\Big)(t,x) \chi_{\omega}^{2}(t,x) 
                 \thinspace \times \\
                 e^{-i \int_{-\infty}^{\frac{1}{2} ( t+\omega\cdot x )} 
		        \big( A_{0} + \sum_{j=1}^{n} \omega_{j}A_{j}\big)(t'+s,x'+s\omega) \thinspace \mathrm{d}s } 
			\thinspace \mathrm{d}x\thinspace \mathrm{d}t
          + \cdots \Bigg|
    \end{multline*}
    where ``$\cdots$'' represents terms that, when divided by $k$, go to zero as $k\to +\infty$. 
    
    Dividing the above expression by $k$ and taking the limit as 
    $k\to\infty$ we get, after using the triangle inequality and
    performing the change of coordinates $(t,x) = \sigma(1,\omega) + Y^{\prime}$
    with $Y^{\prime} \in \Pi_{(1,\omega)}$ 
    \begin{multline}\label{eqn:57}
       \Bigg| \int_{\Pi_{(1,\omega)}} 
       \int_{\mathbb{R}}
          \Big( A_{0} + \sum_{j=1}^{n} \omega_{j}A_{j}\Big)\big(Y^{\prime} + \sigma(1,\omega)\big) 
	        \chi_{\omega}^{2}(Y^{\prime}) \thickspace \times \\
                 e^{-i \int_{-\infty}^{\sigma} 
		        \big( A_{0} + \sum_{j=1}^{n} \omega_{j}A_{j}\big)
			      \big(Y^{\prime} + s(1,\omega) \big) \thinspace \mathrm{d}s } 
			\thinspace \mathrm{d}\sigma\thinspace \mathrm{d}S_{Y^{\prime}}
       \Bigg| \le
       C(n,\Omega) 
       %\Big[ 
       |||\Lambda_{1}-\Lambda_{2} |||. 
       %\Big],
    \end{multline}

    If we set
    \begin{equation*}
    a(Y^{\prime}) := 
       \int_{\mathbb{R}}
          \Big( A_{0} + \sum_{j=1}^{n} \omega_{j}A_{j}\Big)\big(Y^{\prime} + \sigma(1,\omega)\big) 
	         \thinspace e^{-i \int_{-\infty}^{\sigma} 
		        \big( A_{0} + \sum_{j=1}^{n} \omega_{j}A_{j}\big)
			\big(Y^{\prime} + s(1,\omega) \big) \thinspace \mathrm{d}s } 
			\thinspace \mathrm{d}\sigma,
    \end{equation*}
    equation (\ref{eqn:57}) can be rewritten as
    \begin{equation*}
       \Big| \int_{\Pi_{(1,\omega)}} a(Y^{\prime}) \chi^{2}(Y^{\prime}) 
            \thinspace \mathrm{d}S_{Y^{\prime}} \Big| 
       \le
       C(n,\Omega) |||\Lambda_{1}-\Lambda_{2} |||.
    \end{equation*}

    On the other hand the conditions imposed on the support of $\chi$ guarantee
    that the above estimate holds true for any such function satisfying the 
    condition 
    $\int_{\Pi_{(1,\omega)}} |\chi(Y^{\prime})|^{2}\mathrm{d}S_{Y^{\prime}} \le 1$,
    thus $a$ is a bouded linear functional on $L^{1}(\Pi_{(1,\omega)})$ and the estimate
    \begin{multline*}
       \Bigg| \int_{-\infty}^{\infty} 
           \big( A_{0} + \sum_{j=1}^{n} \omega_{j}A_{j}
	   \big) \big( X^{\prime} + \sigma (1,\omega) \big)
             \thickspace \times \\
	     e^{i \int_{-\infty}^{\sigma} 
                  ( A_{0} + \sum_{j=1}^{n} \omega_{j}A_{j} )
	          ( X^{\prime} + s (1,\omega))
                \thinspace \mathrm{d} s
	        }
             \thinspace \mathrm{d} \sigma
       \Bigg| \le
       C(n,\Omega) ||| \Lambda_{1} - \Lambda_{2} |||
    \end{multline*}
    holds. The Fundamental Theorem of Calculus then gives (in the original coordinate system)
    \begin{equation}\label{eqn:58}
       \Bigg|
       \exp \Big[ 
	      i \int_{-\infty}^{\infty} 
           \big( A_{0} + \sum_{j=1}^{n} \omega_{j}A_{j}
	   \big) ( t + s, x + s\omega )
             \thinspace \mathrm{d} s 
            \Big] - 1
       \Bigg| \le
       C(n,\Omega) ||| \Lambda_{1} - \Lambda_{2} |||.
    \end{equation}

     In the case of uniqueness of the potentials it was very easy 
     to go from an expression concerning the above complex exponential
     to an expression involving only the ray transform of the function
     $ A_{0} + \sum_{j=1}^{n} \omega_{j} A_{j} $. In this case, obtaining  
     such an estimate is slightly harder and we need to assume that the 
     following condition holds:
     \begin{itemize}
        \item[\textit{iii)}] the supremum
           \begin{equation*}
 	     \alpha := 
 	     \mathrm{sup}_{(t,x;\omega)\in Q\times S^{n-1}}
 	         \Big| \int_{-\infty}^{\infty} \big( A_{0} + \sum_{j=0}^{n} \omega_{j} A_{j} \big) 
 		        (t+s,x+s\omega) \thinspace\mathrm{d}s \Big|
           \end{equation*}
 	  satisfies the inequality $\alpha < 2\pi $. 
     \end{itemize}
     Denoting by $\beta$ the integral 
     $\int_{-\infty}^{\infty} ( A_{0} + \sum_{j=1}^{n} \omega_{j} A_{j} )  
     (t+s,x+s\omega) \thinspace\mathrm{d}s $ 
     we get 
     \begin{equation} \label{eqn:59}
        \frac{\big| e^{i\beta} -1 \big|}{|\beta|} = 
	\frac{| \sin \frac{\beta}{2} |}{ \frac{|\beta|}{2}}.
     \end{equation}
     Also, condition $iii)$ gives $\dfrac{|\beta|}{2} < \dfrac{\alpha}{2} <\pi $ 
     and the right hand side of (\ref{eqn:59}) is bounded from below by a positive 
     constant $\dfrac{1}{C_{4}}$. We then have the estimate
    \begin{equation*}
       \Big| \int_{-\infty}^{\infty} 
           ( A_{0} + \sum_{j=1}^{n} \omega_{j}A_{j} ) 
	   ( t + s , x + s\omega )
             \thinspace \mathrm{d} s 
	     \Big| \le C_{4}
       \Big| e^{ i \int_{-\infty}^{\infty} 
           ( A_{0} + \sum_{j=1}^{n} \omega_{j}A_{j} ) 
	   ( t + s , x + s\omega )
             \thinspace \mathrm{d} s }
             - 1 \Big|,
    \end{equation*}
    which together with (\ref{eqn:58}) gives
    \begin{equation}\label{estim:abs:ray:integrals}
       \Big| \int_{-\infty}^{\infty} 
           ( A_{0} + \sum_{j=1}^{n} \omega_{j}A_{j} ) 
	   ( t + s , x + s\omega )
             \thinspace \mathrm{d} s 
	     \Big| \le 
       C(n,\Omega) ||| \Lambda_{1} - \Lambda_{2} |||.
    \end{equation}

    We next want to use (\ref{estim:abs:ray:integrals}) as well as the divergence
    condition imposed on the potentials to obtain an estimate for the potentials 
    $A_{j}$, $j=0,\dots,n$, following the ideas in Begmatov's paper \cite{Begmatov}. 
    
    If we let $F$ denote the ray transform of 
    $A_{0} + \sum_{j=1}^{n} \omega_{j} A_{j}$ along 
    light rays, we have 
    \begin{align}\nonumber
    \label{defn:function:F}
    F & : \ERRE{}_{t}\times\ERRE{n}_{x}\times S^{n-1} \to \ERRE{} \\
       F(t,x;\omega) & := 
          \int_{\ERRE{}} \big(A_{0} + \sum_{j=1}^{n} \omega_{j} A_{j} \big) 
	                 (t+s,x+s\omega) \thinspace\mathrm{d}s.
    \end{align}
    and by (\ref{estim:abs:ray:integrals})
    \begin{equation}\label{estimate:F}
       |F(t,x;\omega)| \le C(n,\Omega) ||| \Lambda_{1} - \Lambda_{2} |||
    \end{equation}
    for all $(t,x)\in\ERRE{}_{t}\times\ERRE{n}_{x}$, $\omega\in S^{n-1}$.

    The Fourier transform of $F$ in the variables $x_{1},\dots,x_{n}$ is 
    \begin{equation*}
     \big(\mathcal{F}_{(x\to\xi)}F(t,\cdot;\omega)\big)(\xi) = 
          \int_{\ERRE{n}} 
          e^{-i\xi\cdot x} 
	  \int_{\ERRE{}} \Big(A_{0} + \sum_{j=1}^{n} \omega_{j} A_{j} \Big) 
	                 (t+s,x+s\omega) \thinspace\mathrm{d}s
	              \thinspace \mathrm{d}x.
    \end{equation*}
    and the change of coordinates $\tilde{x}=x+s\omega$, $\tilde{t}=t+s$, with
    Jacobian 
     $\big| \frac{ \partial (\tilde{t},\tilde{x}) }{ \partial (t,x)} \big| = 1$
     leads to
    \begin{align*}
     \big(\mathcal{F}_{(x\to\xi)}F(t,\cdot;\omega)\big)(\xi)
       & = e^{-i(\omega\cdot\xi)t} 
          \int_{\ERRE{n}} 
	  \int_{\ERRE{}} 
	                 e^{-i\tilde{x}\cdot\xi}
			 \thinspace e^{-i(-\omega\cdot\xi)\tilde{t}} 
	                 \Big(A_{0} + \sum_{j=1}^{n} \omega_{j} A_{j} \Big) 
	                 (\tilde{t},\tilde{x}) \thinspace\mathrm{d}\tilde{t}
          \thickspace 
	              \mathrm{d}\tilde{x},
    \end{align*}
    where the right hand side of the above equation is the Fourier
    transform (in all variables) of $A_{0}+\sum_{j=1}^{n}\omega_{j}A_{j}$
    at the point $(-\omega\cdot\xi,\xi)$. This equation can be rewritten
    as
    \begin{equation*}
	e^{it\omega\cdot\xi}
        \big(\mathcal{F}_{(x\to\xi)}F(t,\cdot;\omega)\big)(\xi) =
	\Big(A_{0} + \sum_{j=1}^{n} \omega_{j} A_{j} \Big)^{\wedge} (-\omega\cdot\xi,\xi) 
    \end{equation*}
    and we realize that since the right hand side is independent of
    $t$, so must be the left hand side. In particular when $t=0$ we have
    \begin{equation}\label{Fourier:trans:indep:t}
	\Big(A_{0} + \sum_{j=1}^{n} \omega_{j} A_{j} \Big)^{\wedge} (-\omega\cdot\xi,\xi) =
        \big(\mathcal{F}_{(x\to\xi)}F(0,\cdot;\omega)\big)(\xi) =:
	G(\xi;\omega).
    \end{equation}
    Since the potentials $A_{j}$ are smooth and compactly supported,
    %the function 
    $F(0,\cdot;\cdot):\ERRE{n}_{x}\times S^{n-1}\to \ERRE{}$
    is also smooth and compactly supported\footnote{This is because
    for $|x|$ big enough the light rays with direction $(1,\omega)$ 
    emanating from the point $(0,x)$ do not intersect the support of
    the potentials $A_{j}$.}, 
    moreover (\ref{estimate:F}) shows that it is uniformly bounded by
    $C(n,\Omega)||| \Lambda_{1} - \Lambda_{2} |||$, hence
    \begin{align} \nonumber
	|G(\xi ; \omega) |
	& =  \Big|   \int_{\ERRE{n}} \thickspace
	             e^{-ix\cdot\xi} \thinspace F(0,x;\omega) \thinspace \mathrm{d}x \Big| \\ \nonumber
        & \le
	|| F(0,\cdot;\cdot)||_{L^{\infty}(\ERRE{n}_{x}\times S^{n-1})} \thinspace 
	\mathrm{Vol}(B_{n}(R)) \\ \label{estim:fourier:complement:cone}
	& \le
	    C(n,\Omega) R^{n} \thinspace |||\Lambda_{1} - \Lambda_{2} ||| ,
    \end{align}
    which tells us that $G$ is uniformly bounded in $\ERRE{n}_{\xi}\times S^{n-1}$.

    We now turn our attention to the Fourier transform of the potentials.
    We want to obtain an estimate for $|\widehat{A_{j}}(\tau,\xi)|$ on a conic set
    whose complement contains the light cone 
    $\{(\tau,\xi)\thinspace : \thinspace |\tau| < |\xi| \}$ and use this
    estimate as well as an analytic continuation argument to obtain 
    bounds for $|\widehat{A_{j}}(\tau,\xi)|$ in the full space 
    $\ERRE{}_{\tau}\times\ERRE{n}_{\xi}$.

    For $(\tau,\xi)$ fixed with $|\tau|<\frac{1}{2}|\xi|$ we know by
    considerations made in the previous section that we can find 
    unit vectors $\omega = \omega(\tau,\xi)$ parametrized by $rS^{n-2}$
    (an $(n-2)$-dimensional sphere with radious $r$, $\frac{\sqrt{3}}{2} \le r \le 1$),
    such that 
    $\tau + \omega(\tau,\xi)\cdot\xi = 0$ 
    and satisfying 
    $\omega(\theta\tau,\theta\xi) =  \omega(\tau,\xi)$  
    for any $\theta > 0$, i.e., $\omega(\tau,\xi)$ is homogenous
    of degree $0$ in $(\tau,\xi)$.

    We consider a maximal one dimensional sphere with radious $r$
    contained in $rS^{n-2}$ and choose unit vectors
    $\omega^{(1)}(\tau,\xi),\dots,\omega^{(n)}(\tau,\xi)$ 
    forming the vertices of a regular polygon with $n$ sides. We then 
    consider the following set of $n+1$ equations
    \begin{equation}\label{system:inhomo}
    \left\{
       \begin{aligned}
       \widehat{A_{0}}(\tau,\xi) + 
          \sum_{j=1}^{n} \omega_{j}^{(k)}(\tau,\xi)\widehat{A_{j}}(\tau,\xi) 
	  & = G\big(\xi;\omega^{(k)}(\tau,\xi)\big) 
	  \quad\quad k =1,\dots,n \\
       \frac{1} {\sqrt{\tau^{2}+|\xi|^{2}}} \Big( \tau \widehat{A_{0}}(\tau,\xi) 
       + 
          \sum_{j=1}^{n} 
       \xi_{j}\widehat{A_{j}}(\tau,\xi) \Big) & = 0,
       \end{aligned}
    \right.
    \end{equation}
    where the last equation is a simple consequence of the divergence condition
    $\partial_{t}A_{0}(t,x) + \sum_{j=1}^{n}\partial_{x_{j}}A_{j}(t,x) = 0$.
    Our goal is to show that this system is uniquely solvable for
    $(\widehat{A_{0}},\widehat{A_{1}},\dots,\widehat{A_{n}})$. 
    
    In order to prove this statement it suffices to show that the matrix
    \begin{equation*}
    M(\tau,\xi) = 
       \begin{pmatrix} 
           1 & \omega_{1}^{(1)}(\tau,\xi) & \dots & \omega_{n}^{(1)}(\tau,\xi) \\
           1 & \omega_{1}^{(2)}(\tau,\xi) & \dots & \omega_{n}^{(2)}(\tau,\xi) \\
           \hdotsfor[2]{4}\\
           1 & \omega_{1}^{(n)}(\tau,\xi) & \dots & \omega_{n}^{(n)}(\tau,\xi) \\
           \frac{\tau}{\sqrt{\tau^{2}+|\xi|^{2}}} & 
           \frac{\xi_{1}}{\sqrt{\tau^{2}+|\xi|^{2}}} 
	   & \dots & 
           \frac{\xi_{n}}{\sqrt{\tau^{2}+|\xi|^{2}}} 
       \end{pmatrix}
    \end{equation*}
    is invertible. Notice that the entries of $M(\tau,\xi)$ are homogeneous
    of degree $0$ in $(\tau,\xi)$ and the inverse, if it exists, will also
    have entries that are homogeneous of degree $0$ in $(\tau,\xi)$.

    To see that $M(\tau,\xi)$ is indeed invertible we show the homogeneous
    system
    \begin{equation}\label{system:homo}
    \left\{
       \begin{aligned}
       \widehat{A_{0}}(\tau,\xi) + 
          \sum_{j=1}^{n} \omega_{j}^{(k)}(\tau,\xi)\widehat{A_{j}}(\tau,\xi) 
	  & = 0 
	  \quad\quad k =1,\dots,n \\
       \frac{1} {\sqrt{\tau^{2}+|\xi|^{2}}} \Big( \tau \widehat{A_{0}}(\tau,\xi) 
       + 
          \sum_{j=1}^{n} 
       \xi_{j}\widehat{A_{j}}(\tau,\xi) \Big) & = 0,
       \end{aligned}
    \right.
    \end{equation}
    has no non-trivial solution.
    Once again, the considerations made in the previous section guarantee
    that the only potentials $\mathcal{A} = (A_{0},A_{1},\dots,A_{n})$
    satisfying the first $n$ equations are those of the form
    $\widehat{A_{0}}(\tau,\xi) = \Phi(\tau,\xi) \tau$,
    $\widehat{A_{j}}(\tau,\xi) = \Phi(\tau,\xi) \xi_{j}$, $j=1,\dots,n$,
    for some smooth function $\Phi$. The last equation in the above system 
    leads to $\Phi(\tau,\xi)\sqrt{\tau^{2} + |\xi|^{2}} = 0$ which in turn
    gives $\Phi \equiv 0$ and $\widehat{\mathcal{A}} = 0$.

    Since $M(\tau,\xi)$ is invertible we can write
    \begin{equation*}
       \widehat{A_{j}}(\tau,\xi)  = 
          \sum_{k=1}^{n} c_{k,j}(\tau,\xi) G\big(\xi;\omega^{(k)}(\tau,\xi)\big),\quad
	   1\le k \le n,\thickspace 0 \le j \le n,
    \end{equation*}
    for some $c_{k,j}(\tau,\xi)$ homogeneous of degree $0$ in $(\tau,\xi)$.
    This homogeneity property as well as the uniform boundedness of $G$ allows 
    us to compute an estimate for the Fourier transform of the potentials $A_{j}$
    in the ray $\{(\alpha\tau,\alpha\xi) \thinspace :\thinspace \alpha\in\ERRE{} \}$
    \begin{align}\nonumber
       \big| \widehat{A_{j}}(\alpha\tau,\alpha\xi) \big| 
          & = 
          \Big| \sum_{k=1}^{n} c_{k,j}(\alpha\tau,\alpha\xi) 
	        G\big(\alpha\xi;\omega^{(k)}(\alpha\tau,\alpha\xi)\big) \Big| \\ \nonumber
          & \le 
          \sum_{k=1}^{n} \big| c_{k,j}(\tau,\xi)  \big| 
	                 \big| G\big(\alpha\xi;\omega^{(k)}(\alpha\tau,\alpha\xi)\big) \big| \\ \label{non:uniform}
          & \le 
          C(n,\Omega) |||\Lambda_{1} - \Lambda_{2} ||| \sum_{k=1}^{n} \big| c_{k,j}(\tau,\xi)  \big|, 
    \end{align}
    where in the last line of the previous inequality we used (\ref{estim:fourier:complement:cone}).

    At this point it is convenient to recall that our initial goal is to obtain 
    a uniform bound for $\widehat{A_{j}}(\tau,\xi)$ in the set 
    $\{(\tau,\xi) \thinspace : \thinspace |\tau| \le \frac{|\xi|}{2}\}$.
    In view of (\ref{non:uniform}) it suffices to work on the compact set
    $\{(\tau,\xi) \thinspace : \thinspace \tau^{2}+|\xi|^{2}=1,
    \thickspace|\tau| \le \frac{|\xi|}{2}\}$.
    To obtain such a bound it is necessary to study the entries $c_{k,j}(\tau,\xi)$ 
    of the inverse of the matrix $M(\tau,\xi)$. It is a well know result that such 
    entries can be described by
    \begin{equation*}
       c_{k,j}(\tau,\xi) = \frac{1}{\det M(\tau,\xi)} \mathrm{C}_{j,k}(\tau,\xi)
    \end{equation*}
    where $\mathrm{C}_{j,k}(\tau,\xi)$ is the $(j,k)$-cofactor of $M(\tau,\xi)$.

    Since on the set 
    $\{(\tau,\xi) \thinspace : \thinspace \tau^{2}+|\xi|^{2}=1,
    \thickspace|\tau| \le \frac{|\xi|}{2}\}$
    all the entries of $M(\tau,\xi)$ have absolute value less or equal to one, 
    and since $\mathrm{C}_{j,k}(\tau,\xi)$ consists of sums of products of 
    $n$ such entries, we have
    \begin{equation*}
       | c_{k,j}(\tau,\xi) | 
       \le \frac{| \mathrm{C}_{j,k}(\tau,\xi) |}{| \det M(\tau,\xi) |} 
       \le \frac{n}{| \det M(\tau,\xi) |}.
    \end{equation*}

    The quantity 
    $|\det M(\tau,\xi)|$ 
    can be interpreted as the $(n+1)$-dimensional
    volume generated by the  set of vectors 
    $\{(1,\omega^{(1)}(\tau,\xi)),\dots,(1,\omega^{(n)}(\tau,\xi)),(\tau,\xi) \}$,
    however, since independent of the values $(\tau,\xi)$ the vectors
    $ \omega^{(1)}(\tau,\xi),\dots, \omega^{(n)}(\tau,\xi)$ 
    are chosen to be the vertices of a regular polygon, the $n$-dimensional
    volume $\mathrm{V}$ generated by 
    $\{(1,\omega^{(1)}(\tau,\xi)),\dots,(1,\omega^{(n)}(\tau,\xi))\}$,
    is constant. We then have 
    $|\det M(\tau,\xi)| = V \times \mathrm{P}(\tau,\xi)$ 
    where $\mathrm{P}(\tau,\xi)$ is the projection of $(\tau,\xi)$
    into the linear subspace generated by 
    $\{(1,\omega^{(1)}(\tau,\xi)),\dots,(1,\omega^{(n)}(\tau,\xi))\}$.
    This projection is given by $C\sin \varphi$ where $\varphi$
    is the angle between $(\tau,\xi)$ and said subspace. Since
    the vectors $(1,\omega^{(k)}(\tau,\xi))$, $k=1,\dots,n$, 
    are located in the boundary of the light cone (i.e., the 
    set 
    $\{(\tau,\xi) \thinspace : \thinspace |\tau| = |\xi| \}$),
    this angle is bounded below by $\frac{\pi}{8}$. Therefore
    on the set 
    $\{(\tau,\xi) \thinspace : \thinspace \tau^{2}+|\xi|^{2}=1,
    \thickspace|\tau| \le \frac{|\xi|}{2}\}$ the value
    $|\det M(\tau,\xi)|$ is uniformly bounded from below by
    $\mathrm{V} \sin \frac{\pi}{8}$ and
    \begin{equation*}
       | c_{k,j}(\tau,\xi) | 
       \le \frac{n}{V \sin \frac{\pi}{8}}.
    \end{equation*}
    
    These observations combined with (\ref{non:uniform}) give the uniform
    estimate
    \begin{align} \label{eq:estim:Aj}
       \big| \widehat{A_{j}}(\tau,\xi) \big| 
          & \le 
          C(n,\Omega) |||\Lambda_{1} - \Lambda_{2} ||| 
    \end{align}
    on the set 
    $\{(\tau,\xi) \thinspace : \thinspace |\tau| \le \frac{|\xi|}{2}\}$. 
%
    %which in combination with (\ref{estim:fourier:complement:cone}) lead to the estimate
    %\begin{align} \nonumber
    %   |\widehat{A_{j}} (\tau,\xi) | 
    %      & \le C n \max_{1\le k \le n} |G(\tau,\xi;\omega^{(k)})| 
%	    \qquad\text{where } C = 
%	    \max_{\substack{ 0\le j \le n \\ 1 \le k \le n }} c_{k,j} \\ \label{eq:estim:Aj}
%           %\le C n ||| \Lambda_{1} - \Lambda_{2} |||
%          & \le C(n,\Omega) ||| \Lambda_{1} - \Lambda_{2} |||
%    \end{align}
%    for any $(\tau,\xi)$ satisfying $|\tau| < |\xi|$ 
%    (i.e., outside of the light cone).

    Our next step is to obtain an upper bound for 
    $ \widehat{A_{j}}(\tau,\xi)$, $j=0,\dots,n,$
    on the complement of 
    $\{(\tau,\xi) \thinspace : \thinspace |\tau| \le \frac{|\xi|}{2}\}$. 
    To obtain such estimate we first fix $\tau$ and compute upper bounds for all lines
    that pass through the origin and are contained in the hyperplane $\tau=\tau_{0}$. 
    %We do so by means of an analytic continuation argument as in \cite{Begmatov},

    We will consider the case where the line corresponds to the $\xi_{n}$-axis,
    but before doing so, we will need some auxiliary results.

    \begin{lemma}
    Consider the strip 
    \begin{equation*}
    S = \{ z = z_{1} + iz_{2} \thinspace: \thinspace z_{1}\in\ERRE{}, 
           |z_{2}|< 2|\tau_{0}|\pi, \tau_{0}\neq 0\}
    \end{equation*}
    and the rays
    \begin{equation*}
     p_{1} = \{ z \thinspace: \thinspace -\infty < z_{1} \le -2|\tau_{0} |, z_{2}=0\}, \qquad
     p_{2} = \{ z \thinspace: \thinspace 2|\tau_{0}| \le z_{1} < \infty, z_{2}=0\} 
    \end{equation*}
    in the complex plane $\mathbb{C}$. \\
    If $E = p_{1} \cup p_{2}$ and $G=S\setminus E$ is the 
    strip with cuts along the rays $p_{1}$ and $p_{2}$, we have
    \begin{equation}\label{estim:harmon:measure}
       \frac{2}{3} < \varpi(z,E,G) \le 1,
    \end{equation}
    where $\varpi(z,G,E)$ is the harmonic measure of $E$ with respect to $G$.
    \end{lemma}
    This statement is a very well known result about harmonic measures, its
    proof is mostly taken from \cite{Begmatov} and it is included here for 
    the purpose of self contention.
    \begin{proof}
        For $h>0$, the map 
	\begin{equation*}
	   \zeta(z) = \left( \frac{\exp(z\pi/h) - \exp(a\pi/h)}
	                          {\exp(z\pi/h) - \exp(-a\pi/h) } \right)^{\frac{1}{2}}
	\end{equation*}
	comformally transforms $\overline{G}$ into the upper half plane
	$\mathbf{H}^{+} = \{ \zeta = \zeta_{1} + i \zeta_{2} \thinspace : \thinspace 
	\zeta_{1} \in \ERRE{}, \zeta_{2} \ge 0 \}$. Under this mapping,
	the interval $I = \{ z\thinspace :\thinspace |z_{1}|<a, z_{2}=0\}$
	transforms into the imaginary half axis $\{\zeta \thinspace : 
	\thinspace \zeta_{1} = 0, \zeta_{2} >0 \}$. The boundary of $G$
	goees into the real axis and the set $E = p_{1} \cup p_{2}$ 
	tranforms into the subset of the real axis
	\begin{equation*}
	   E_{1} = 
	   \{ \zeta \thinspace : \thinspace \zeta_{1} \le -e^{a\pi/h}, \zeta_{2} =0 \} \cup
	   \{ \zeta \thinspace : \thinspace |\zeta_{1}| \le 1, \zeta_{2} =0 \} \cup
	   \{ \zeta \thinspace : \thinspace \zeta_{1} \ge e^{a\pi/h}, \zeta_{2} =0 \}.
	\end{equation*}
	Then by the harmonic principle (see \cite{Bukhgeim16}), the values of the harmonic
	measures on $E$, $E_{1}$ with respect to the sets $G$, $\mathbf{H}^{+}$ agree, this is
	\begin{equation*}
	   \varpi(z,E,G) = \varpi(\zeta(z),E_{1},\mathbf{H}^{+}).
	\end{equation*}
	We also know that the harmonic measure on the right hand side of the 
	previous equation can be constructed by means of the Poisson integral
	for the upper half plane
	\begin{equation}\label{eqn:Poisson:int}
	   \varpi(\zeta) = \frac{1}{\pi} \int_{-\infty}^{\infty}
	                 \chi_{E_{1}}(t) \frac{\zeta_{2}}{(t-\zeta_{1})^{2}+\zeta_{2}^{2}}
			 \thinspace\mathrm{d}t
	\end{equation}
	where $\chi_{E_{1}(t)}$ is the characteristic function of $E_{1}$.

	Since we are interested in the image of $I$ under the map $\zeta(z)$ and since
	this image is precisely the positive imaginary axis, we may assume without loss
	of generality that $\zeta= i \zeta_{2}$, $\zeta_{2}>0$. From 
	(\ref{eqn:Poisson:int}) we obtain
	\begin{equation}\label{harmonic:general}
	     \varpi(\zeta(z),E_{1},\mathbf{H}^{+}) = \frac{2}{\pi} 
	     \left(
	        \frac{\pi}{2} -\arctan \frac{\zeta_{2} (\exp(a\pi/h) -1 ) }{(\zeta_{2})^{2} + \exp(a\pi/h)}
	     \right)
	\end{equation}
	Choosing $h=a\pi$ and using the inequality 
	\begin{equation}
	   \arctan \frac{\zeta_{2} ( e -1 ) }{(\zeta_{2})^{2} + e } \le
	   \arctan \frac{ (e -1 ) }{2 e^{\frac{1}{2}} } 
	\end{equation}
	we obtain
	\begin{equation*}
	     \frac{2}{\pi}\left(
	        \frac{\pi}{2} -\arctan \frac{ (e -1 ) }{2 e^{\frac{1}{2}}} 
	     \right)
		\le \varpi \le 1
	\end{equation*}
	and we conclude that
	\begin{equation*}
	     \frac{2}{3} \le \varpi(z,E,G) \le 1.
	\end{equation*}
    \end{proof}
    
    Based on this result we want to `embed' the $\xi_{n}$-axis into said strip and use the 
    bounds on the harmonic measure. To do so we realize that since the potentials $A_{j}$,
    $j=0,\dots,n$, are 
    compactly supported, the functions
    $  \widehat{A_{j}} (\tau_{0},\xi) $ 
    admits an analytic extension in $\xi_{n}$ into the complex plane. If we let
    \begin{gather*}
    \Pi = \{ \nu = (\nu_{1},\nu_{2}) \thinspace: \thinspace \nu_{1}\in\ERRE{}, 
           |\nu_{2}|< 2|\tau_{0}|\pi, \thinspace \tau_{0} \neq 0 \}, \\
     q_{1} = \{ \nu = (\nu_{1},\nu_{2}) \thinspace: \thinspace -\infty < \nu_{1} \le -2|\tau_{0} |, \nu_{2}=0\}, \\
     q_{2} = \{ \nu = (\nu_{1},\nu_{2}) \thinspace: \thinspace 2|\tau_{0}| \le \nu_{1} < \infty, \nu_{2}=0\} 
    \end{gather*}
    and restrict ourselves to the $\xi_{n}$-axis (i.e., $\xi_{1}=\cdots=\xi_{n-1}=0$),
    the estimate (\ref{estim:harmon:measure}) leads to
    \begin{equation*}
       \frac{2}{3} < \varpi(\nu,E_{1},G_{1}) \le 1,
    \end{equation*}
    where $E_{1} = q_{1} \cup q_{2}$ and $G_{1}=\Pi\setminus E_{1}$.

    Denoting by  $ v_{j}(\nu) = \widehat{A_{j}}
       (2\tau_{0},0,\dots,0,\nu), $
    %$v_{j}(\nu)$ 
    the restriction of $\widehat{A}_{j}$ to the $\xi_{n}$-axis, 
    %this is 
    we have by the two-constant theorem (see \cite{Krantz} Theorem 9.4.5)
    \begin{equation}\label{two:constant:ineq}
        |v_{j}(\nu) | \le m_{j}^{\frac{2}{3}} M_{j}^{\frac{1}{3}} 
    \end{equation}
    where $m_{j}$ and $M_{j}$ are the respective upper bounds of the modulus of
    $v(\nu)$ on the rays $q_{1}$ and $q_{2}$ and on the lines
    $\{ (\nu_{1},\nu_{2}) \thinspace : \thinspace \nu_{1}\in\ERRE{}, \nu_{2} = -2|\tau_{0}|\pi \}$
    and
    $\{ (\nu_{1},\nu_{2}) \thinspace : \thinspace \nu_{1}\in\ERRE{}, \nu_{2} =  2|\tau_{0}|\pi \}$.

    At this point it is worth to point out that the rays $q_{1}$ and $q_{2}$ 
    are contained in the set 
    $\{(\tau,\xi) \thinspace:\thinspace |\tau | \le \frac{|\xi|}{2} \}$ 
    and that we have already 
    computed an estimate
    for $|v_{j}(\nu)|$ in that region (equation \ref{eq:estim:Aj}). 
    To compute $M_{j}$ we resort to the 
    identities
    \begin{equation*}
       v_{j}(\nu) =  C(\pi,n) \int_{\ERRE{}} 
              e^{-i (\nu_{1}+i\nu_{2}) x_{n}} 
	      W_{j}(2\tau_{0},0,\dots,0,x_{n}) \thinspace \mathrm{d}x_{n}
    \end{equation*}
    with $W_{j}$ the Fourier transform of $A_{j}$ in all variables except $x_{n}$.
    Next, we realize that these functions are compactly supported in $x_{n}$ and the above
    integrand is nonzero only on a finite subset of the real numbers. Hence
    \begin{equation*}
       |v_{j}(\nu) | \le \mathrm{sup}_{x_{n}\in(-a(\Omega),a(\Omega))} 
                          | W_{j}(2\tau_{0},0,\dots,0,x_{n}) |
                     \int_{-a(\Omega)}^{a(\Omega)} 
		     e^{2|\tau_{0}| \pi x_{n}} \mathrm{d}x_{n},
    \end{equation*}
    where $a(\Omega)$ is a positive number bigger than $\mathrm{diam}(\Omega)$.
    Since $\nu = \nu_{1} + i\nu_{2}$ is restricted to the strip $\Pi$ we 
    have
    \begin{equation*}
       |v_{j}(\nu) | \le C(\Omega,n) \frac{e^{2|\tau_{0}|\pi}}{ 2|\tau_{0}| \pi},
    \end{equation*}
    and in particular when 
    $\nu$ is a real number satisfying $-2|\tau_{0}| < \nu < 2|\tau_{0}|$ we have by
    (\ref{two:constant:ineq})
    \begin{equation*}
       |v_{j}(\nu) | \le C(\Omega,n) \frac{e^{\frac{2|\tau_{0}|\pi }{3} } 
                         m_{j}^{\frac{2}{3} }}{ 2|\tau_{0}|^{\frac{1}{3}} }.
    \end{equation*}

    All this arguments can be carried out to the case where a line is contained in 
    the hyperplane $\tau = \tau_{0}$, passes through the origin but is not parallel 
    to any of the axes.
    Finally using (\ref{eq:estim:Aj}) 
    we obtain an estimate for the Fourier
    transform of the potentials in the set
    $\{(\tau,\xi) \thinspace:\thinspace |\tau | > \frac{|\xi|}{2} \}$,
    %interior of the light cone, 
    namely
    \begin{equation} \label{estim:fourier:cone}
    |\widehat{A_{j}}(\tau,\xi)|
        \le C(\Omega,n) \frac{ e^{ \frac{2|\tau|\pi}{3} } 
	|||\Lambda_{1}-\Lambda_{2}|||^{\frac{2}{3} }  }
	{ |\tau|^{\frac{1}{3}} }.
    \end{equation}
    %whenever $(\tau,\xi)\in LC:= \{ |\tau| \ge |\xi| \}$.
    
     From estimates (\ref{eq:estim:Aj}) and (\ref{estim:fourier:cone}) we can 
     establish the desired 
     estimate for the vector potentials. The general idea is to 
     use the inequality 
    $||f||_{L^{\infty}} \le || \widehat{f} \thinspace ||_{L^{1}}$ 
    and partition $\ERRE{}_{\tau}\times\ERRE{n}_{\xi}$ in 
    an appropriate way.

    From the Fourier inversion formula we have
    \begin{equation}\label{eqn:star}
       A_{j}(t,x) = 
          C(\pi) \iint_{\ERRE{}_{\tau}\times\ERRE{n}_{\xi}} 
          e^{i(t\tau + x\cdot\xi)} \widehat{A_{j}}(\tau,\xi)
	  \thinspace\mathrm{d}\tau\mathrm{d}\xi
    \end{equation}
    and by taking absolute values we have for any $\rho_{1}>0$
    \begin{align*}
       %\Big| \Big( A_{0} + \sum_{j=1}^{n} \omega_{j}A_{j}\Big)(t,x) \Big| 
       \big| A_{j}(t,x) \big|
       & \le \thinspace C(\pi) \iint_{\ERRE{}_{\tau}\times\ERRE{n}_{\xi}} 
            \big| \thinspace 
	    %\Big( A_{0} + \sum_{j=1}^{n} \omega_{j}A_{j}\Big)^{\wedge}
	    \widehat{A_{j}}(\tau,\xi)
	    \thinspace\big|
	    \thinspace\mathrm{d}\tau\mathrm{d}\xi \\
       & \le \thinspace C(\pi) \iint_{B(\rho_{1})} 
            \big| \thinspace 
	    %\Big( A_{0} + \sum_{j=1}^{n} \omega_{j}A_{j}\Big)^{\wedge}
	    \widehat{A_{j}}(\tau,\xi)
	    \thinspace\big|
	    \thinspace\mathrm{d}\tau\mathrm{d}\xi \\
       & \phantom{\le} + \thinspace C(\pi) \iint_{B(\rho_{1})^{c}} 
            \big| \thinspace 
	    %\Big( A_{0} + \sum_{j=1}^{n} \omega_{j}A_{j}\Big)^{\wedge}(\tau,\xi)
	    \widehat{A_{j}}(\tau,\xi)
	    \thinspace\big|
	    \thinspace\mathrm{d}\tau\mathrm{d}\xi \\
      & = I_{1} + I_{2},
    \end{align*}
    where $B(\rho_{1})$ denotes the $(n+1)$-dimensional ball
    $B(\rho_{1})= \{ (\tau,\xi) \thinspace : \thinspace 
    |\tau|^{2} + |\xi|^{2} \le \rho_{1}^{2}\}$.

    To obtain a bound for $I_{2}$ we recall that the potentials 
    $A_{j}$, $j=1,\dots,n$, are $C^{\infty}$ in $t$ and $x$. Hence for 
    any $\beta > 0$ and any 
    $\rho_{1}>0$ if $|\tau|^{2} + |\xi|^{2} \le \rho_{1}^{2}$
    \begin{equation*}
       \big| 
       %(A_{0} + \sum_{j=0}^{n} \omega_{j}A_{j})^{\wedge}(\tau,\xi) 
       \widehat{A_{j}}(\tau,\xi)
       \big| 
       \le \frac{C}
       { (|\tau|^{2} + |\xi|^{2} )^{\frac{\beta}{2}} }.
    \end{equation*}
    If $\beta > n+2$ the integral $I_{2}$ converges. Moreover, the estimate
    \begin{equation}\label{estability:one:over:rho}
       I_{2} = \iint_{B(\rho_{1})^{c}} 
       \big| 
       % (A_{0} + \sum_{j=0}^{n} \omega_{j}A_{j})^{\wedge}(\tau,\xi) 
       \widehat{A_{j}}(\tau,\xi)
       \big| 
          \mathrm{d}\tau\mathrm{d}\xi
       \le \frac{C(n)}
       { \rho_{1}^{\beta - n - 1} }
       \le \frac{C(n)}
       { \rho_{1} }
    \end{equation}
    holds.

    To estimate $I_{1}$ we break up the ball $B(\rho_{1})$ into two smaller pieces
    \begin{equation*}
     \mathcal{C}_{1} = 
     B(\rho_{1}) \cap \Big\{ (\tau,\xi) \thinspace: \thinspace |\tau| < \frac{\rho_{1}}{\sqrt{5}} \Big\}
     \quad \text{and} \quad
     \mathcal{C}_{2} = 
     B(\rho_{1}) \cap \Big\{ (\tau,\xi) \thinspace: \thinspace |\tau| \ge \frac{\rho_{1}}{\sqrt{5}} \Big\}.
    \end{equation*}
    %as in \textbf{figure ?}.

    Then 
    \begin{equation*}
       I_{1} \le 
       \iint_{\mathcal{C}_{1}} 
       \big| 
       %(A_{0} + \sum_{j=0}^{n} \omega_{j}A_{j})^{\wedge}(\tau,\xi) 
          \widehat{A_{j}}(\tau,\xi)
       \big| 
          \mathrm{d}\tau\mathrm{d}\xi +
       \iint_{\mathcal{C}_{2}} 
       \big| 
       %(A_{0} + \sum_{j=0}^{n} \omega_{j}A_{j})^{\wedge}(\tau,\xi) 
          \widehat{A_{j}}(\tau,\xi)
       \big| 
          \mathrm{d}\tau\mathrm{d}\xi,
    \end{equation*}
    and since $\mathcal{C}_{1}$ is a compact subset of $B(\rho_{1})$ we have
    \begin{equation*}
       I_{1} \le C \rho^{n+1} +
       \iint_{C_{2}} 
       \big| 
       %(A_{0} + \sum_{j=0}^{n} \omega_{j}A_{j})^{\wedge}(\tau,\xi) 
          \widehat{A_{j}}(\tau,\xi)
       \big| 
          \mathrm{d}\tau\mathrm{d}\xi.
    \end{equation*}
    The advantage of this decomposition is that $\mathcal{C}_{2}$ is contained in the
    set $\{(\tau,\xi) \thinspace:\thinspace |\tau | > \frac{|\xi|}{2} \}$ and 
    %light cone $LC=\{ (\tau,\xi) \thinspace : \thinspace |\tau| \le |\xi| \}$ and 
    that on this set $|\tau|$ is bounded below by $\dfrac{\rho_{1}}{\sqrt{5}}$. Thus 
    by (\ref{estim:fourier:cone})
    \begin{equation}\label{estability:two:constant}
       I_{2} \le C \rho^{n+1} +
       \frac{C(\Omega,n) \thinspace 
             e^{\frac{2\rho\pi}{3}} 
             ||| \Lambda_{1} - \Lambda_{2} |||^{\frac{2}{3}} \rho^{n+1}}
             { \rho^{\frac{1}{3}}}
    \end{equation}
    where we have set $\rho = \dfrac{\rho_{1}}{\sqrt{5}}$.

    Equations (\ref{eqn:star})-(\ref{estability:two:constant}) 
    lead to
    \begin{align}\label{eqn:two:stars} \nonumber
       \big| 
       %\Big( A_{0} + \sum_{j=1}^{n} \omega_{j}A_{j}\Big)(t,x) 
          A_{j}(t,x)
       \big| 
       & \le \thinspace C(\Omega,n) \Big[
             \frac{1}{\rho} + \rho^{n+1} + 
             \rho^{n+\frac{2}{3}} \thinspace e^{\frac{2\rho\pi}{3}} 
	            ||| \Lambda_{1} - \Lambda_{2} |||^{\frac{2}{3}} 
	     \Big] \\
       & \le \thinspace C(\Omega,n) \Big[
             \frac{1}{\rho} + 
                    \rho^{n+\frac{2}{3}} \thinspace e^{\frac{2\rho\pi}{3}} 
	            ||| \Lambda_{1} - \Lambda_{2} |||^{\frac{2}{3}} 
	     \Big].
    \end{align}

    Now the idea is to choose $\rho$ small enough so that the two terms in the 
    the right hand side of (\ref{eqn:two:stars}) are comparable. In other words
    we want $\rho$ to satisfy the identity
    \begin{equation*}
       \frac{C}{\rho} = 
              \rho^{n+\frac{2}{3}} \thinspace e^{\frac{2\rho\pi}{3}} 
	            ||| \Lambda_{1} - \Lambda_{2} |||^{\frac{2}{3}} 
    \end{equation*}
    for some constant $C$. By taking logarithms on both sides of the previous 
    equation we obtain the equivalent identity
    \begin{equation}\label{eqn:three:stars}
       2 \log \frac{C} { \thinspace ||| \Lambda_{1} - \Lambda_{2} ||| } 
       = (3n+5) \log \rho + 2\pi\rho,
    \end{equation}
    and since the right hand side of (\ref{eqn:three:stars}) is one to one when
    $\rho > 0 $ we know that it admits a unique solution.

    On the other hand, the inequality $\log\rho \le \rho$ for positive $\rho$ 
    as well as (\ref{eqn:three:stars}) lead to
    \begin{equation*}
       2 \log \frac{ C }
	         { \thinspace ||| \Lambda_{1} - \Lambda_{2} ||| } 
       \le (3n+5 + 2\pi)  \rho, 
    \end{equation*}
    or
    \begin{equation*}
       \frac{1}{\rho} \le 
             \frac{3n+5+2\pi}{2} \Bigg[ \log 
             \frac{ C }
	          { ||| \Lambda_{1} - \Lambda_{2} |||  }
       \Bigg]^{-1}
    \end{equation*}
    and equation (\ref{eqn:two:stars}) becomes
    \begin{equation*} 
       \big| 
       %\Big( A_{0} + \sum_{j=1}^{n} \omega_{j}A_{j}\Big)(t,x) 
          A_{j}(t,x)
       \big| 
       \le C(\Omega,n)\Bigg[ \log 
             \frac{ C }
	          { ||| \Lambda_{1} - \Lambda_{2} |||  }
       \Bigg]^{-1}.
    \end{equation*}

    Summarizing, we have proved the following
    \begin{thm}\label{stability:theorem}
       Suppose that the vector potentials 
       $\mathcal{A}^{(l)}=(A^{(l)}_{0},\dots,A^{(l)}_{n})$, $l=1,2,$ 
       are real valued, compactly supported and 
       $C^{\infty}$ in $t$ and $x$. 
       Let $ \mathcal{A} = (A_{0},A_{1},\dots,A_{n})$
       where $ A_{j} = A_{j}^{(1)} - A_{j}^{(2)} $ 
       and suppose that the divergence condition 
       \begin{equation*}
          %\mathrm{div}\thinspace\mathcal{A} = 0 
          \mathrm{div}\thinspace\mathcal{A} = 
	     \partial_{t}A_{0}(t,x) + 
	     \sum_{j=1}^n{} \partial_{x_{j}}A_{j}(t,x) = 0
       \end{equation*}
       holds.
       If $\Lambda_{l}$ represents the Dirichlet to Neumann
       operator associated to the hyperbolic problem (1)-(4), 
       then the stability estimate 
    \begin{equation}\label{eqn:four:stars}
       \max_{0\le j \le n}\Big|\Big| 
       %\Big( A_{0} + \sum_{j=1}^{n} \omega_{j}A_{j}\Big)
       A^{(1)}_{j}(t,x) - A^{(2)}_{j}(t,x) 
       \Big|\Big|_{L^{\infty}(\ERRE{}_{t}\times\ERRE{n}_{x})}
       \le C(\Omega,n)\Bigg[ \log 
             \frac{ C }
	          { ||| \Lambda_{1} - \Lambda_{2} |||  }
       \Bigg]^{-1}
    \end{equation}
    holds for $\Lambda_{1}$, $\Lambda_{2}$ satisfying $|||\Lambda_{1} - \Lambda_{2} ||| << 1$.
    \end{thm}
    \begin{proof}
         The hypothesis of the theorem guarantee that conditions 
	 \textit{i)} and \textit{ii)} hold and the only condition that is not automatically 
	 satisfied from the hypothesis is condition \textit{iii)}. 
	 However a simple rescaling of the vector potentials
	 $A_{j} \to A^{\prime}_{j} = \frac{1}{\alpha}A_{j}$, with $\alpha$ the supremum 
	 of all ray integrals of the potentials as well as a similar rescaling of the 
	 coordinate axis show that the estimate (\ref{eqn:four:stars}) holds for 
	 the potentials $A^{\prime}_{j}$. In turn, this implies a similar estimate for
	 the original potentials $A_{j}$.
    \end{proof}

%\section{Presence of obstacles}\label{sec:proofs:part2}
\section{Presence of obstacles}\label{sec:proofs:part2}

%  \begin{itemize}
%  \item[1)]
     One variation of the above problem consists in the introduction of
     convex obstacles inside the domain $\Omega$. That is, let $\Omega_{k}$,
     $1\le k \le M$, be simply connected bounded domains in $\mathbb{R}^{n}$,
     $n \ge 3$ and let $D=\Omega \setminus \cup_{k=1}^{M}\Omega_{k}$.
     If we consider again the equation
     \begin{equation*}
     \HYPEQN{A}{V}{u} = 0 \qquad\text{in }\ERRE{}\times\Omega,
     \end{equation*}
    \begin{align*}
     u(t,x) 
     %= \partial_{t}u(t,x) 
     & = 0\makebox[2.5 em]{} \qquad\text{for }
               t<<0 \\ \label{hyp:eqn1b}
                          u(t,x) & = f\makebox[2.5 em]{$(t,x)$}\qquad\text{on }
              \ERRE{}\times\partial\Omega,
     \end{align*}
     with the additional condition
     \begin{equation}
         u(t,x)\big|_{\ERRE{}\times\partial\Omega_{k}}= 0,\quad 1\le k \le M.
     \end{equation}

     Then we can prove as in the case of no obstacles that the vector valued potential
     $\mathcal{A}(t,x)$ satisfies

     \begin{equation}\label{eqn:45deg:rays}
        \mathcal{PA}(t,x;\omega) =
                    \int_{\infty}^{\infty} \sum_{j=0}^{n}\omega_{j}
                    A_{j}(t+s,x+s\omega)\thinspace\mathrm{d}s = 0
     \end{equation}
     for any ray
     \begin{equation}\label{ray:gamma}
      \gamma =  \{(t,x)+s(1,\omega) \thinspace | \thinspace s\in \mathbb{R} \}
     \end{equation}
     not intersecting the obstacles $\ERRE{}\times\Omega_{k}$, $1\le k \le M$.
     We will show that under some conditions the vector and
    scalar potentials can be recovered outside these obstacles.

    As a warm up let us consider the case where $n=3$, there is only
    one obstacle and the integrals over light rays are zero for a
    % formula (\ref{eqn:45deg:rays}) holds for a 
    smooth
    scalar function $f=f(t,x)$ satisfying the support condition $f(t,x)=0$
    when $|x| > R$, and the growth condition $ | f(t,x) | \le C(1+|t|)^N$
    for some integer $N$.

    Under these settings, for $\tilde{x}$ an arbitrary point in $D$,
    we can find a two dimensional plane $P$ in $\mathbb{R}^{3}$ 
    containing $\tilde{x}$ such
    that $P$ misses $\Omega_{1}$, this is $P\cap \Omega_{1} = \emptyset$.
    We then realize that the set of rays $\gamma$ of the form
    (\ref{ray:gamma}) that are contained in the three dimensional
    space $\mathbb{R}^{1}_{t}\times P$ and pass through $\tilde{x}$,
    can be parametrized by the one dimensional sphere $S^{1}$. This
    in turn, allows us to use our previous considerations for a scalar 
    function in two dimensions (c.f. \emph{Ramm}--\emph{J. Sj\"ostrand}, 
    \cite{Ramm:Sjostrand}) to conclude that
    $f$ vanishes in $\mathbb{R}^{1}_{t}\times P$. Since $P$ could
    be any two dimensional plane not intersecting $\Omega_{1}$ and
    $\tilde{x}$ was selected to be an arbitrary point in
    $D =\Omega \setminus \Omega_{1}$ we conclude that $f=0$
    on $D\times \mathbb{R}^{1}_{t}$.

    Unfortunately for the vector potential things are going to be
    slightly harder and we will have to make some geometric considerations
    in order to obtain an equivalent result.

    We shall impose the following restriction on the problem:
    \begin{itemize}
       \item[\emph{(G1)}] The obstacles are convex and
                   \begin{itemize}
                        \item when $n\ge4$, for each $\tilde{x}
                              \in D$ there exists a two dimensional plane $P$
                  passing through $\tilde{x}$ such that $P$ does not intersect
		  any of the obstacles.
                        \item when $n=3$ there is only one obstacle.
                   \end{itemize}
    \end{itemize}

    \begin{thm}\label{obstacles:1}
    If condition (G1) holds and the Dirichlet to Neumann operators 
    are equal ($\Lambda_{1} = \Lambda_{2}$),
    then the potentials $\mathcal{A}^{(1)}$ and $\mathcal{A}^{(2)}$
    are gauge equivalent, i.e., there exists a smooth function $\varphi$
    such that $\mathcal{A}^{(1)} - \mathcal{A}^{(2)} = \mathrm{d}\varphi$.
    \end{thm}
    \begin{proof}
     Let us proceed first by assuming that $n=3$.

     Regarding the vector potential $\mathcal{A}(t,x)$ as a 1-form, the
     goal will be to prove that the 2-form $d\mathcal{A}$ vanishes
     outside $\Omega_{1}$.
     %and conclude by using a standard argument from
     %vector calculus.

     For $\tilde{x}$ not in $\Omega_{1}$, using condition (G1) we can find
     a 2-dimensional plane $P \subset \mathbb{R}^{3}_{x}$ not intersecting
     the obstacle and we can choose three linearly independent unit vectors
     $\eta_{j}$, $j=1,2,3$ close to $P$ such that any plane $P_{jp}$,
     $1\le j,p\le 3$ passing through $\tilde{x}+\eta_{j}$ and $\tilde{x}+\eta_{p}$
     does not meet $\Omega_{1}$. Next, introducing coordinates
     $x^{\prime}=(x^{\prime}_{1},x^{\prime}_{2},x^{\prime}_{2})$ in $\mathbb{R}^{3}_{x}$
     by the formula $x - \tilde{x} = x_{1}^{\prime}\eta_{1}+ \cdots +x_{3}^{\prime}\eta_{3}$
     and denoting by $\mathcal{A}^{\prime}$ and $B^{\prime}$ the 1-form and 2-form
     mentioned above expressed in the new coordinate system, we have
      \begin{align}\label{diff:forms}
          \mathcal{A}^{\prime} & = A^{\prime}_{0}\mathrm{d} x_{0}^{\prime} +
                      A^{\prime}_{1}\mathrm{d} x_{1}^{\prime} +
                          A^{\prime}_{2}\mathrm{d} x_{2}^{\prime} +
              A^{\prime}_{3}\mathrm{d} x_{3}^{\prime} \\
          B^{\prime}(t,x^{\prime}) & = \sum_{0\le p<j\le3} b^{\prime}_{jp}(t,x^{\prime})
                              \thinspace \mathrm{d}x_{j}^{\prime}
                  \wedge \mathrm{d}x_{p}^{\prime}
      \end{align}
      where $x_{0}^{\prime}= t$ and
      \begin{equation}\label{b:jays}
          b^{\prime}_{jp} = \frac{\partial\thinspace A^{\prime}_{j} }{\partial x_{p}^{\prime}}   -
                            \frac{\partial\thinspace A^{\prime}_{p} }{\partial x_{j}^{\prime}}.
      \end{equation}
     The restriction of $B^{\prime}$ to the 3-dimensional space
     $\ERRE{1}_{t}\times P_{jp}$ is given by
     \begin{equation}\label{restric:B}
         B^{\prime}\big|_{\ERRE{1}_{t}\times P_{jp}} =
           b^{\prime}_{0j} \thinspace\mathrm{d}x_{0}^{\prime}
            \wedge \mathrm{d}x_{j}^{\prime} +
           b^{\prime}_{0p} \thinspace\mathrm{d}x_{0}^{\prime}
            \wedge \mathrm{d}x_{p}^{\prime} +
           b^{\prime}_{jp} \thinspace\mathrm{d}x_{j}^{\prime}
            \wedge \mathrm{d}x_{p}^{\prime}
     \end{equation}
     and since there are no obstacles in $\ERRE{1}_{t}\times P_{jp}$
     theorem (\ref{prop:uniq2}) gives
     \begin{equation}\label{grad:A:new:coord}
         A^{\prime}\big|_{\ERRE{1}_{t}\times P_{jp}} = \nabla
              \varphi_{jp}(t,x_{j}^{\prime},x_{p}^{\prime})
     \end{equation}
      where $\nabla$ is the gradient in the $(t,x_{j}^{\prime},x_{p}^{\prime})$ and
      $\varphi_{jp}$ is a smooth function 
      compactly supported in $x_{jp}^{\prime}=(x^{\prime}_{j},x^{\prime}_{p})$.
      %that vanishes when
      %$ (x_{j}^{\prime})^{2} + (x_{p}^{\prime})^{2}$ is big enough.
      Clearly then 
      equation (\ref{b:jays}) gives
       $B^{\prime}\big|_{\Pi_{jp}\times\mathbb{R}_{t}^{1}} = 0$.

     As the above discussion holds for all three dimensional
     spaces parallel and close to $P_{jp}$ we see that $B^{\prime}$ and
     thus $B = \mathrm{d}\mathcal{A}$ vanishes for $x$ near $\tilde{x}$ and for
     all values of $t$. Being that
     $\tilde{x}$ is an arbitrary point not in the obstacle, we get that
     $\mathrm{d}\mathcal{A}$ vanishes outside $\Omega_{1}$ and thus,
     since we are working in a simply connected domain,
     the vector potential is
     the gradient of a smooth function.

      When there are two or more obstacles and $\tilde{x}$ is a point not lying in any of them,
      we resort to the geometric condition (G1) to find a two dimensional plane
      $P_{0}$ such that $P_{0}$ intersects no obstacles and choose
      $n$ linearly independent unit vectors $\eta_{j},\dots,\eta_{n}$ close to
      $P_{0}$ in such a way that $\eta_{1},\eta_{2}\in P_{0}$ and the planes
      $P_{jp}$, $1\le j,p \le n$, do not intersect any obstacle. 

      As before, we introduce new coordinates $x^{\prime}=
      (x^{\prime}_{1},\dots,x^{\prime}_{n})$ given by $x - \tilde{x} =
      x_{1}^{\prime}\eta_{1}+ \cdots +x_{n}^{\prime}\eta_{n}$, and express
      $\mathcal{A}$ and $B$ in terms of these new coordidates.
      If we now consider the restrictions of $B$  to the spaces
      $\ERRE{1}_{t}\times P_{jp}$, $1\le j,p \le n$, we realize thet we are 
      back into the previous case. Making use of
      the fact that $\mathcal{A}^{\prime}\big|_{\ERRE{1}_{t}\times P_{jp}} =
      \nabla \varphi_{jp}(t,x^{\prime}_{j},x^{\prime}_{p})$  we obtain
      via equation (\ref{b:jays}) that
       \begin{equation}\label{B:restric:sigma}
          B^{\prime}\big|_{\ERRE{1}_{t}\times P_{jp}} = 0,
       \end{equation}
       which as before leads to $B=0$ and hence
       $ \mathcal{A}(t,x) = \nabla_{t,x}\Psi(t,x)$
       for some smooth $\Psi$.
       \end{proof}

%\section{Presence of obstacles (Part 2)} \label{sec:proofs:part3}
%
%  \item[2)]
    When condition \textit{(G1)} fails to hold, we can still recover some
    information regarding the difference of the vector potentials provided
    that $\mathrm{d}\mathcal{A}=0$. 
    %we have an appropriate geometric constraint. 
    As we did before, let us regard vector potentials as $1$-forms
    %$(\mathcal{A} = A_{0}\mathrm{d}t+A_{1}\mathrm{d}x_{1}+\cdots
    %   +A_{n}\mathrm{d}x_{n})$ 
    and let us consider the case when we have 2 obstacles inside the domain 
    $\Omega$. 
    %and the vector
    %potential satisfies the condition 

    %Since 
    This time 
    %the domain is not simply connected 
    it might be the case that
    the domain is not simply connected and that 
    for some close path $\gamma$ the integral
    $\int_{\gamma} \mathcal{A}\cdot\mathrm{d}\bar{x}$
    is not zero (here we set $\bar{x}=(t,x)$), however, the condition
    $\mathrm{d}\mathcal{A}=0$ guarantees that it only depends on the
    homotopy class of $\gamma$.

    With this consideration in mind we want to be able to ``span'' every
    possible homotopy class representative $\gamma$ of the domain $\Omega$
    by using light rays. In other words, we would like to impose on
    our domain the geometric condition $(G2)$: every homotopy 
    representative can be continuosly contour-deformed into light rays.

    If this geometric condition is met, we can write for $\gamma_{1}$ an
    arbitrary simple closed path in $\Omega$ surrounding the first of the
    obstacles
    \begin{equation*}
        \int_{\gamma_{1}} \mathcal{A}(\bar{x})\cdot\thinspace\mathrm{d}\bar{x} =
        \int_{\ell_{1}} \mathcal{A}(\bar{x})\cdot\thinspace\mathrm{d}\bar{x} +
        \int_{\ell_{2}} \mathcal{A}(\bar{x})\cdot\thinspace\mathrm{d}\bar{x} + \cdots +
        \int_{\ell_{r}} \mathcal{A}(\bar{x})\cdot\thinspace\mathrm{d}\bar{x} ,
    \end{equation*}
    where the set of light rays $\ell_{1},\ell_{2},\dots,\ell_{r}$, surround the
    first of the obstacles and are such that they do not intersect the second
    obstacle. Then by the previous arguments regarding the construction of
    geometric optics solutions and Green's formula we have if
    $\mathcal{A}^{(1)}$ and $\mathcal{A}^{(2)}$ correspond to equal
    Dirichlet to Neumann operators
    \begin{equation*}
        \int_{\gamma_{1}} \left(\mathcal{A}^{(1)}(\bar{x}) -
    \mathcal{A}^{(2)}(\bar{x})\right) \cdot\thinspace\mathrm{d}\bar{x} = 2\pi m_{1}.
    \end{equation*}
    Proceeding in a similar fashion, we have for the second obstacle and a contour
    $\gamma_{2}$ surrounding it
    \begin{equation*}
        \int_{\gamma_{2}} \left(\mathcal{A}^{(1)}(\bar{x}) -
    \mathcal{A}^{(2)}(\bar{x})\right) \cdot\thinspace\mathrm{d}\bar{x} = 2\pi m_{2}.
    \end{equation*}
    Then, for an arbitrary closed contour $\gamma$ we have (after a contour deformation
    argument) that
    \begin{equation*}
        \int_{\gamma=c_{1}\gamma_{1}+c_{2}\gamma_{2}}
         \left(\mathcal{A}^{(1)}(\bar{x}) - \mathcal{A}^{(2)}(\bar{x})\right)
         \cdot\thinspace\mathrm{d}\bar{x}
         = 2\pi c_{1}m_{1} + 2\pi c_{2}m_{2}.
    \end{equation*}
    where $c_{1}$ and $c_{2}$ are two integers. We now let $\Theta(\bar{x})$,
    be the function that computes the angle between
    the projection of $\bar{x}$ into the hyperplane $t=0$ and the vector $(0,0,\dots,0,1)$,
    and for $j=1,2$ we set $\Theta_{j}(\bar{x}) = \Theta(x-p_{j})$, where $p_{j}$
    is any point inside the obstacle $j$.
    Then the functions $\Theta_{j}$, compute the `angle that a vector makes inside
    the obstacle $j$' and we have
    \begin{equation*}
        \int_{\gamma}
         \left(\mathcal{A}^{(1)}(\bar{x}) - \mathcal{A}^{(2)}(\bar{x})
               - m_{1}\Theta_{1}(\bar{x}) - m_{2}\Theta_{2}(\bar{x}) \right)
         \cdot\thinspace\mathrm{d}\bar{x} = 0
    \end{equation*}
    for any closed contour $\gamma$. Therefore
    \begin{equation}\label{eqn:varphi}
         \mathcal{A}^{(1)}(\bar{x}) - \mathcal{A}^{(2)}(\bar{x})
               - m_{1}\Theta_{1}(\bar{x}) - m_{2}\Theta_{2}(\bar{x})
         = \partial_{\bar{x}} \varphi (\bar{x})
    \end{equation}
    for some function $\varphi$.

    We can certainly say more, if $\gamma(\bar{x}_{0};\bar{x})$ is any curve
    joining the points $\bar{x}=(t,x)$ and $\bar{x}_{0}=(t_{0},x_{0})$, and  we
    let
    \begin{equation*}
        C_{0}(\bar{x}) = \exp \left(
                     -i \int_{\gamma(\bar{x}_{0};\bar{x})} \left(
             m_{1}\Theta_{1}(\bar{x}) - m_{2}\Theta_{2}(\bar{x})
             \right) \cdot\thinspace\mathrm{d}\bar{x} \right),
    \end{equation*}
    then
    \begin{equation*}
        \frac{i}{C_{0}(\bar{x})} \partial_{\bar{x}} C_{0}(\bar{x})
    = m_{1}\Theta_{1}(\bar{x}) - m_{2}\Theta_{2}(\bar{x})
    \end{equation*}
    and equation (\ref{eqn:varphi}) can be reewritten as
    \begin{equation*}
         \mathcal{A}^{(1)}(\bar{x}) - \mathcal{A}^{(2)}(\bar{x}) -
        \frac{i}{C_{0}(\bar{x})} \partial_{\bar{x}} C_{0}(\bar{x}) =
    \partial_{\bar{x}} \varphi (\bar{x})
    \end{equation*}
    or
    \begin{equation*}
         \mathcal{A}^{(1)}(\bar{x}) - \mathcal{A}^{(2)}(\bar{x}) =
        \frac{i}{C(\bar{x})} \partial_{\bar{x}} C(\bar{x})
    \end{equation*}
    where $C(\bar{x})=C_{0}(\bar{x})\exp(i\varphi(\bar{x}))$.

    To conclude this section let us consider the case when 
    we have any number of obstacles,
    $\mathrm{d}\mathcal{A} = 0$, 
    $ \mathcal{A} = \mathcal{A}^{(1)} - \mathcal{A}^{(2)} $,
    $\Omega$ is multi-connected and condition $(G2)$ holds.

    \begin{thm}\label{obstacles:2}
    If $\mathrm{d}\mathcal{A}^{(1)} - \mathrm{d}\mathcal{A}^{(2)} = 0$,
    condition (G2) is satisfied and the Dirichlet to Neumann operators 
    are equal ($\Lambda_{1} = \Lambda_{2}$).
    Then the potentials $\mathcal{A}^{(1)}$ and $\mathcal{A}^{(2)}$
    are gauge equivalent, 
    i.e., there exists a smooth function $C=C(\bar{x})$
    such that $|C| = 1$ for 
    $\bar{x} = (t,x) \in \ERRE{}\times\partial\Omega$ and
    $\mathcal{A}^{(1)} - \mathcal{A}^{(2)} = 
     \frac{i}{C(\bar{x})} \partial_{\bar{x}} C(\bar{x})$.
    \end{thm}
    \begin{proof}
    Let $G\big((-T,T)\times\Omega\big)$ denote the gauge group 
    corresponding to the set $(-T,T)\times\Omega$ and let 
    $\ell_{1},\dots,\ell_{p}$ be a basis for the homology group, 
    this is, $\gamma = n_{1}\ell_{1} + \cdots + n_{p}\ell_{p}$ 
    for any closed contour $\gamma$. Since $\mathrm{d}\mathcal{A} = 0$ 
    the integrals 
    $\int_{\gamma} \mathcal{A}(\bar{x})\cdot\mathrm{d}\bar{x}$
    depend only on the homotopy class of $\gamma$. By condition
    $(G2)$ the basis of the homology group can be spanned by light 
    rays that do not intersect any of the obstacles, hence if 
    $\mathcal{A}^{(1)}$ and $\mathcal{A}^{(2)}$ correspond to 
    gauge equivalent Dirichlet-to-Neumann operators, we have 
    as before that 
    \begin{equation*}
       e^{ i \int_{\ell_{k}} \mathcal{A}^{(1)}(\bar{x}) \cdot \mathrm{d}\bar{x}} = 
       e^{ i \int_{\ell_{k}} \mathcal{A}^{(2)}(\bar{x}) \cdot \mathrm{d}\bar{x}}.
    \end{equation*}
    Therefore for any closed curve $\gamma$
    \begin{equation*}
        \int_{\gamma} \big( \mathcal{A}^{(1)}(\bar{x}) - \mathcal{A}^{(2)}(\bar{x}) \big)  
	                \cdot \mathrm{d}\bar{x} =  2\pi q,\quad q\in\mathbb{Z},
    \end{equation*}
    which in turn implies the existance of a function
    $C(t,x)\in G\big( (-T,T)\times\Omega \big)$
    such that 
    \begin{equation*}
         \mathcal{A}^{(1)}(\bar{x}) - \mathcal{A}^{(2)}(\bar{x}) =
        \frac{i}{C(\bar{x})} \partial_{\bar{x}} C(\bar{x}).
    \end{equation*}
    \end{proof}
%  \end{itemize}

%%%%%%%%%%%%%%%%%%%%%%%%%%%%%%%%%%%%%%%%%%%%%%%%%%%%%%%%%%%%%
% APPENDIX
%%%%%%%%%%%%%%%%%%%%%%%%%%%%%%%%%%%%%%%%%%%%%%%%%%%%%%%%%%%%%
\pagebreak
\section*{Appendix A: Auxiliary results}
\subsubsection*{Orthogonal complement of a set in $\ERRE{m}$}
%%%%%%%%%%%%%%%%%%%%%%%%%%%%%%%%%%%%%%%%%%%%%%%%%%%%%%%%%%%%%
% APPENDIX
%%%%%%%%%%%%%%%%%%%%%%%%%%%%%%%%%%%%%%%%%%%%%%%%%%%%%%%%%%%%%
%\section*{Appendix A: Auxiliary results}
%\chapter{
%Auxiliary results}
%\subsubsection*{
%Orthogonal complement of a set in $\ERRE{m}$}
%We want to prove the following

\textit{ \textbf{Lemma: }
The orthogonal complement of the set
\begin{equation*}
   E=\{(1,\omega) \thinspace : \thinspace \omega\in S^{(n-1)},
   \tau +\omega\cdot\xi=0, |\tau|<|\xi|\}
\end{equation*}
is a one dimensional subspace of $\ERRE{n+1}$.  }

Let us start off with some linear algebra facts.

Let $m\ge 2$. If $A\subseteq B \subseteq \ERRE{m}$ with
\begin{equation*}
A\subseteq B \subseteq \mathrm{Span}(A) =
\Big\{
\sum_{p=1}^{r} \alpha_{p}a_{p} \thinspace : \thinspace
\alpha_{p}\in\ERRE{}, a_{p}\in A, r\in\mathbb{N}
\Big\},
\end{equation*}
then
\begin{equation*}
A^{\perp} =
B^{\perp} =
\mathrm{Span}(A)^{\perp}.
\end{equation*}

Indeed, since orthogonal complements reverse inclusions we have
$\mathrm{Span}(A)^{\perp}\subseteq B^{\perp} \subseteq A^{\perp}$.
Also, if $x\cdot a = 0$ for all $a\in A$, then
$x\cdot\sum_{p=0}^{r}\alpha_{p}a_{p} =
\sum_{p=0}^{r}\alpha_{p}(x\cdot a_{p}) = 0 $,
which shows that $A^{\perp}\subseteq \mathrm{Span}(A)^{\perp}$.
Therefore $A^{\perp} = B^{\perp} = \mathrm{Span}(A)^{\perp}$.

Denoting by $\mathrm{CH}(A)$ the \emph{convex hull} of $A$ and by
$\mathcal{C}(A)$ the \emph{cone spanned} by $A$ we have
\begin{align*}
\mathrm{CH}(A) & = \Big\{
\sum_{p=1}^{r} \alpha_{p}a_{p} \thinspace : \thinspace
\sum_{p=1}^{r} \alpha_{p}=1, 0\le \alpha_{p}\le 1, a_{p}\in A, r\in\mathbb{N}
\Big\} \\
\mathcal{C}(A) & = \{ ta \thinspace : \thinspace t\in\ERRE{+}, a\in A\}.
\end{align*}
Clearly $A\subseteq \mathrm{CH}(A) \subseteq \mathcal{C}\big(\mathrm{CH}(A)\big)$
and since both sets contain particular linear combinations of elements of $A$ we
also have
   $\mathrm{Span}\Big(\mathcal{C}\big(\mathrm{CH}(A)\big) \Big)
   = \mathrm{Span}(A)$,
hence,
\begin{equation*}
   \mathrm{Span}\Big(\mathcal{C}\big(\mathrm{CH}(A)\big) \Big)^{\perp}
   = \mathrm{Span}(A)^{\perp}.
\end{equation*}

We want to apply these remarks to the set $E$ but before doing so let
us recall that for $|\tau|<|\xi|$ the vectors $\omega$ satisfaying
$|\omega|=1$, $\tau+\omega\cdot\xi = 0$, can be parametraized by $S^{n-2}$.
Since rotations are non-singular transformations we can compute instead the
dimension of the orthogonal complement of the set
\begin{equation*}
   \tilde{E} =
   \{(1,\omega_{1},\dots,\omega_{n-1},a) \thinspace : \thinspace
   \omega_{1}^{2}+\cdots+\omega_{n-1}^{2} = 1 - a^{2}\},
\end{equation*}
where $0\le a < 1$ is a fixed number. Taking into account our previous observations
we then have
\begin{align*}
   \tilde{E}^{\perp} & = \mathrm{Span}\Big( \mathcal{C}\big( \mathrm{CH}(\tilde{E})\big)\Big)^{\perp} \\
   & = \mathrm{Span}\Big(\mathcal{C}\big(\{(1,\omega_{1},\dots,\omega_{n-1},a) \thinspace : \thinspace
   \omega_{1}^{2}+\cdots+\omega_{n-1}^{2} \le 1 - a^{2}\}\big)\Big)^{\perp}, \\
   & = \mathrm{Span}\big(\{(t,t\theta_{1},\dots,t\theta_{n-1},ta) : 
   %t\in\ERRE{},\theta=(\theta_{1},\dots,\theta_{n-1})\in\ERRE{n-1}, |\theta|^{2} \le 1 - a^{2}\}\big)^{\perp}
   t,\theta_{1},\dots,\theta_{n-1}\in\ERRE{}, \theta_{1}^{2} +\cdots +\theta_{n}^{2}\le 1 - a^{2}\}\big)^{\perp}
\end{align*}
and we can see that $\tilde{E}^{\perp}$ is a one dimensional subspace of
$\ERRE{n+1}$ since clearly
$\mathrm{Span}\Big(\mathcal{C}\big(\mathrm{CH}(\tilde{E})\big)\Big)$ is
$n$-dimensional.

%We want to prove the following

%%%%%%%%%%%%%%%%%%%%%%%%%%%%%%%%%%%%%%%%%%%%%%%%%%%%%%%%%%%%%
% REFERENCES
%%%%%%%%%%%%%%%%%%%%%%%%%%%%%%%%%%%%%%%%%%%%%%%%%%%%%%%%%%%%%
\pagebreak


\begin{thebibliography}{99}
     \bibitem{Ballesteros:Weder} Ballesteros, M., Weder, R. \emph{High-velocity estimates for the 
     		   scattering operator and Aharonov-Bohm effect in three dimensions.}  Comm. Math. 
		   Phys. 285  (2009), No. 1, 345--398.
     \bibitem{Begmatov} Begmatov, Akram Kh. \emph{A certain inversion problem for the ray transform
                   with incomplete data}. Siberian Mathematical Journal, Vol. 42,
           No. 3, pp. 428--434, 2001.
     \bibitem{Belishev} Belishev, M. \emph{Boundary control in reconstruction of manifolds and metrics 
                   (the BC method)}. Inverse Problems, 13 (1997) R1--R45.
     \bibitem{Bukhgeim16} Bukhgeim, A. L. \emph{On one class of Volterra operator equations of the
                   first kind}. Funktsional. Anal. i Prilozhen., 6, No. 1, 1--9 (1972).
     \bibitem{Eskin:approach} Eskin, Gregory. \emph{A new approach to hyperbolic inverse problems}.  
                   Inverse Problems,  22  (2006),  no. 3, 815--831.
     \bibitem{Eskin:approach2} Eskin, Gregory. \emph{A new approach to hyperbolic inverse problems II. 
                   Global step}. Inverse Problems,  23  (2007),  no. 6, 2343--2356.
     \bibitem{Eskin:time:dependent} Eskin, Gregory. \emph{Inverse hyperbolic problems with
                   time-dependent coefficients}, Comm. Partial Differential Equations, Vol. 32,
                   No. 11, pp. 1737--1758, 2007.
     \bibitem{Eskin:obstacles} Eskin, Gregory. \emph{Inverse problems for the Schr\"odinger equations 
                   with time-dependent electromagnetic potentials and the Aharonov-Bohm effect.}
		   J. Math. Phys.  49  (2008),  no. 2, 022105, 18 pp.
     %\emph{Optical Aharonov-Bohm Effect: An Inverse 
      %             Hyperbolic Problems Approach}. Commun. Math. Phys. 284 (2008), 317--343.
     %\bibitem{Eskin:elliptic} Eskin, Gregory. \emph{Inverse problems for the Schr\"odinger operators
     %              with electromagnetic potentials in domains with obstacles}, Inverse Problems 19 (2003)
     %              985--996.
     %\bibitem{Eskin:broken:rays} Eskin, Gregory. \emph{Inverse boundary value problems in domains with several
     %              obstacles}, Inverse Problems, 20 (2004) 1497--1516.
     \bibitem{Finch} Finch, David V. \emph{Cone beam reconstruction with sources on a curve}.
                   SIAM J. APPL. MATH. Vol. 45, No. 4, pp. 665--673, August 1985.
     \bibitem{Helgason} Helgason Sirgudur. \emph{The Radon transform (Progress in Mathematics)}.
                   Birkh\"auser Boston; 2nd ed. edition, August 1, 1999.
     \bibitem{Hartog} H\"ormander, Lars. \emph{An introduction to complex analysis in several variables}.
                   D. Van Nostrand Co. Princeton, N.J.-Toronto, Ont.-London, 1966.
     \bibitem{Hormander} H\"ormander, Lars. \emph{The analysis of Linear Partial Differential 
                   Operators I}. Springer, 1985.
     \bibitem{Isakov:uniq:stab} Isakov, Victor. \emph{Uniqueness and stability in multi-dimensional 
                   inverse problems.} Inverse Problems 9 (1993), no. 6, 579--621.
     \bibitem{Isakov}  Isakov, Victor. \emph{Inverse problems for partial differential equations},
                   Applied Mathematical Sciences 127, Springer-Verlag New York Inc., 1991.
     \bibitem{Isakov:Sun} Isakov, V., Sun, Z. Q. \emph{Stability estimates for hyperbolic inverse 
                   problems with local boundary data.} Inverse Problems 8 (1992), no. 2, 193--206.
     \bibitem{Krantz} Krantz, Steven G. \emph{Geometric Function Theory. Explorations in complex analysis}.
                   Cornerstones, Birkh\"auser Boston, 2006.
     \bibitem{Kurylev} Kurylev, Y. \emph{Multi-dimensional inverse boundary problems by BC-method: 
                   groups of transformations and uniqueness results}. Math. Comput. Modelling, 
		   18 (1993) 33--45.
     \bibitem{Kurylev:Lassas} Kurylev, Y., Lassas, M. \emph{Hyperbolic inverse problems with data on 
                   a part of the boundary}. AMS/1P Stud. Adv. Math., 16 (2000) 259--72.
     \bibitem{Natterer} Natterer, F. \emph{Mathematics of computarized tomography}, Wiley, New York, (1986).
     \bibitem{Nicoleau} Nicoleau Fran\c{c}ois. \emph{An inverse scattering problem with the Aharonov-Bohm
                   effect}. Journal of Mathematical Physics. Vol. 41, No. 8. pp. 5223--5237. August 2000.
     \bibitem{Palamodov} Palamodov, V. P. \emph{Reconstruction from ray integrals with sources on
                   a curve}. Inverse Problems 20 (2004) 239--242.
     \bibitem{Radon} Radon, J. \emph{\"Uber die Bestimmung von Funktionen durch ihre Integralwerte
                   l\"angs gewisser Mannigfaltigkeiten}. Ber. Verh. Sachs. Akad. Wiss.
           Leipzig, Math-Nat. K1., 69 (1917) 262--277.
     \bibitem{Ramm:Sjostrand} Ramm, A. G., Sj\"ostrand. J. \emph{An inverse problem of the wave
                   equation}. Math. Z. 206, 119--130 (1991).
     \bibitem{Schiff} Schiff, L. I. \emph{Quantum Mechanics}, Third Edition, McGraw-Hill, New York, (1955).
     \bibitem{Stefanov} Stefanov, P. \emph{Uniqueness of the multidimensional inverse scattering
                   problem for time dependent potentials}. Math. Z. 201, 541--560 (1989).
     \bibitem{Stefanov:Uhlmann} Stefanov, P., Uhlmann, G., \emph{Stability estimates for the hyperbolic 
                   Dirichlet to Neumann map in anisotropic media}.  J. Funct. Anal., 154 (1998),  no. 2, 
		   330--358.
     %\bibitem{Strichartz} Strichartz, Robert S. \emph{Radon Inversion--Variations on a theme}.
     %              The American Mathematical Monthly, Vol. 89, No. 6, pp. 377--384, (Jun.-Jul., 1982).
     \bibitem{Sylvester:Uhlmann} Sylvester, J., Uhlmann, G. \emph{A global uniqueness theorem for an 
                    inverse boundary value problem.} Ann. of Math. (2)  125  (1987),  no. 1, 153--169.
     \bibitem{Tataru} Tataru, D. \emph{Unique continuation for solutions to PDE}. Commun. Partial Diff. 
                   Eqns., 20 (1995) 855--84.
     \bibitem{Tuy} Tuy, Heang K. \emph{An inversion formula for cone-beam reconstruction}.
                   SIAM J. APPL. MATH. Vol. 43, No. 3, pp.546--552, June 1983.
     \bibitem{Uhlmann} Uhlmann, Gunther. \emph{Inverse scattering in anisotropic media. Surveys
     		   on solution methods for inverse problems}, pp 235--251, Springer, Vienna, 2000.
     \bibitem{Weder} Weder, Ricardo. \emph{The Aharonov-Bohm effect and time-dependent inverse scattering
     		   theory}. Inverse Problems 18 (2002) 1041--1056.
     \bibitem{Weder2} Weder, Ricardo. \emph{High-Velocity Estimates for the Scattering Operator and Aharonov-Bohm
     		   Effect in Three Dimensions}. arXiv: 0711.2569.
\end{thebibliography}
\end{document}